%% file: arxivVersion.tex
\documentclass[11pt]{article}
\usepackage{geometry}                
\usepackage{graphicx,subcaption}
\usepackage{amssymb,amsmath,amsthm}
\usepackage{epstopdf}
\usepackage{personal}
\usepackage{color,xcolor}
\usepackage{mathtools}
\usepackage{tikz}
\usetikzlibrary{arrows,intersections,patterns,calc}

\title{Dominating Points of Gaussian Extremes} 

\author{Harsha Honnappa$^1$\and Raghu Pasupathy$^2$\and Prateek Jaiswal$^1$\\}
\date{
$^1$School of Industrial Engineering, Purdue University\\
$^2$Department of Statistics, Purdue University\\[2ex]
}

\begin{document}
\maketitle

\begin{abstract}
We quantify the large deviations of Gaussian extreme value statistics
on closed convex sets in $\bbR^d$. Specifically, we characterize the
asymptotic decay behavior of $\bbP({\mathbf A_n^{-1}} M_n \in
\mathcal{C})$, where $M_n$ is the componentwise maximum of independent
and identically distributed, mean-zero Gaussian random vectors in
$\bbR^d$, $\mathcal{C} \subset \bbR^d$ is an ``atypical" convex set
with interior but not including the origin, and ${\mathbf A_n}$ is a
certain scaling matrix. The asymptotics imply that $M_n$'s
distribution exhibits a rate function that is a simple quadratic function of a unique ``dominating point" $x^*$ located on the boundary of the convex set $\sC$. Furthermore, $x^*$ is identified as the optimizer of a certain convex quadratic programming problem, indicating a ``collusion" between the dependence structure of the Gaussian random vectors and the geometry of the convex set $\sC$ in determining the asymptotics. We specialize our main result to polyhedral sets which appear frequently in other contexts involving logarithmic asymptotics. We also extend the main result to characterize the large deviations of Gaussian-mixture extreme value statistics on general convex sets. Our results have implications to contexts arising in rare-event probability estimation and stochastic optimization, since the nature of the dominating point and the rate function suggest importance sampling measures.


\end{abstract}



\input{intro_new}
\input{multivariate_v2}


\bibliographystyle{apt} 
\bibliography{refs}   
\end{document}

%% file: intro_new.tex
\section{Introduction}

Motivated by numerous applications in operations research, engineering, and the sciences, we study certain ``large deviation" asymptotics of the {\it componentwise maximum} \begin{equation}\label{blext}M_n := (\max_{1\leq i\leq n} X_i^1, \ldots, \max_{1\leq i\leq n} X_i^d),\end{equation} where $X_i = (X_i^1,X_i^2,\ldots,X_i^d) \in \bbR^d, i=1,2,\ldots$ are independent and identically distributed (iid) Gaussian random vectors. Traditionally, the study of $M_n$ has been the domain of Extreme Value Theory (EVT), where an extensive literature has accumulated beginning with the seminal work of Fisher and Tippett\cite{FiTi1928} and Gnedenko\cite{Gn1943} on collections of univariate random variables, with more recent efforts focusing on multivariate collections; see~\cite{DeFe2007} for a comprehensive review of the main results. There is also substantial work on so-called uniform convergence results that aim to understand how large the collection needs to be before the extreme value effect ``kicks in.'' The treatment in~\cite{DeHo1972} seems to be the earliest on this question, followed by numerous other papers~\cite{Co1982,Co1982b,Ha1979b,HaWe1979,Sm1982}. Chapter 4 of the monograph \cite{Vi1994} and the essay by R.W. Anderson in Chapter C, Section I of~\cite{de1983} summarizes this line of work.

It may be argued that the essential focus of EVT is the \emph{nominal} behavior of the block maximum as the number of elements in the block $n \to \infty$. In contrast to EVT, in this paper, we seek to characterize the likelihood of large exceedances of the maximum when measured on a specified scale. Precisely, we characterize an asymptotic of the form \begin{equation}\label{rldpintro} \lim_{n \to \infty} \frac{1}{v(\|\mathbf A_n)\|)} \log \bbP\left(\mathbf A_n^{-1} M_n \in \sC\right) = -J(x^*),\end{equation} where $\bbP\left(\mathbf A_n^{-1} M_n \in \sC\right) \to 0$, $\mathbf A_n$ are appropriately chosen ``scaling matrices" with $\|\mathbf A_n\| \to \infty$, $\sC \subset \bbR^d$ is any closed convex set with non-empty interior, $J(\cdot)$ is a positive-valued rate function, $x^* \in \sC$ is called the \emph{dominating point} in the parlance of~\cite{Ne1984}, and $v(\cdot)$ is a positive-valued function connoting ``speed." We note that the asymptotic in (\ref{rldpintro}) is more restricted than the usual large deviation principle~\cite{DeZe2010}; for this reason, we adopt the terminology of \cite{DuLeSu2003,KoMa2015} and informally call the asymptotic in (\ref{rldpintro}) a {\it restricted} large deviation principle (rLDP). 

\subsection{Motivation} 
To the best of our knowledge~\cite{GiMa2014,Vi2015}~are among the only few papers that explicitly study rLDPs of this type, and that too specifically in the univariate setting. Apart from being of interest in the same statistical modeling contexts as EVT, we have found that the asymptotic in (\ref{rldpintro}) is of particular usefulness in rare-event estimation~\cite{2007asmgly,JuSh2006}. It is well known that standard Monte Carlo estimators are inefficient for estimating rare event probabilities. Consequently, importance sampling (IS) methods have become popular recently and involve constructing a simulation estimator through an altered sampling measure that is designed to reduce the relative error of the resulting estimator. In the light-tailed case, for instance, the change of measure is typically achieved by an exponential measure change (EMC). Crucial to the effectiveness of an IS method for rare event estimation is the choice of the IS measure, often aided by locating the so called dominating point(s) associated with an underlying large-deviation principle~\cite{JuSh2006,ChPa2018}. (The dominating point in a rare event context should be loosely understood as the point around which the rare event probability of interest concentrates, and thus serves as the natural choice for locating the mode of an IS measure.) 

Much of this cited prior work focuses on estimating the likelihood that a multivariate random walk or multidimensional Markov process enters a rare event set. On the other hand, in many engineering settings it is useful to know that the componentwise maximum of a multivariate state vector has entered a rare event  set of interest; see, for instance, examples from reliability engineering~\cite{NiShNa2001}. Similarly, in machine learning and stochatic optimization, performance evaluation of sample average approximation (SAA)~\cite{kim2015guide,shapiro2013sample} or empirical risk minimization (ERM) algorithms~\cite{vapnik1998statistical} requires an understanding of whether the estimated multivariate optimizer lies in a set of interest. In all of these cases, there is a need to understand the large deviations behavior of the componentwise maximum $M_n$ in~(1), as it enters a rare event set $\sC$. In particular, can an appropriate dominating point be characterized? In this paper we focus on convex rare event sets and Gaussian random vectors (and their mixtures) for simplicity, and answer this in the affirmative.

\subsection{Commentary on Main Results}
Our main result, appearing as the rLPD characterized in Theorem~\ref{thm:5}, yields a number of important insights into the nature of the asymptotic likelihood of large exceedances of $M_n$. First, and paralleling~\cite{Ne1983} for multivariate random walks, Theorem~\ref{thm:5} asserts that  the asymptotic in (\ref{rldpintro}) equals a certain simple quadratic function of a characterizable point $x^*$ that lies on the boundary of the set $\sC$. We call $x^*$ the (information-theoretic) {\it dominating point} since the assertion of Theorem~\ref{thm:5} implies that no other point in the set $\sC$ is needed for characterizing the rate $J(\sC)$, under the scaling implied by $\mathbf A_n$. Second, the nature of $v(\cdot)$ characterized through Theorem~\ref{thm:5} indicates that, under the considered scaling, the tail likelihood asymptotically decays at the same order as a Frech{\' e}t distribution.

We show that the dominating point $x^*$ identified by Theorem~\ref{thm:5} is encoded as the solution of a constrained quadratic optimization program, with the quadratic objective emerging from the level manifolds of the multivariate Gaussian density, and constraints imposed by the closed convex set of interest. The location of $x^*$ is thus a consequence of the ``collusion" between the Gaussian structure and the geometry  of the constraining convex set. Importantly, $x^*$ {\it need not} be the point on the boundary of the convex set that is closest to the origin, contrary to intuition. The presence of a unique dominating point and the ability to easily characterize its location has important implications to rare-event estimation because $x^*$'s location suggests that the rare event $\mathbf A^{-1}M_n \in \sC$ will most likely happen by $\mathbf A_n^{-1} M_n$ ``entering" the set $\sC$ around the point $x^*$.


We specialize Theorem~\ref{thm:5} to specific closed convex sets of interest, starting with the standard block cone in Proposition~\ref{prop:boxes}, followed by polyhedral sets defined using a piecewise affine boundary function; see Figure~\ref{fig:2}. The rLDP for a standard block cone follows as a straightforward special case of the main result. By contrast, and even though polyhedral sets can be expressed using invertible linear transformations of the block cone, an rLDP for polyhedral sets poses a few additional technical hurdles due to the possibility of transformation matrix not being invertible. 

Finally, in Proposition~\ref{prop:gmm}, we characterize an rLDP for a componentwise maximum constructed from Gaussian-mixture random vectors under the condition that none of the mean vectors (or ``centroids") of the mixture components lie in the closed convex set. GMM's are widely used in statistical modeling although model performance analysis is often complicated by the fact that the overall distribution is not Gaussian. The rLDP results reported, on the other hand, provide a means of performance analysis within such models, particularly in understanding rare event outcomes.

\subsection{Technical Details}
The proof of Theorem~\ref{thm:5} requires a number of technical difficulties to be hurdled, and proceeds through a series of lemmas. Lemma~\ref{prop:convex} establishes an rLDP determining the log-likelihood that a single scaled Gaussian random vector enters the closed convex set $\sC$. Our analysis is general, allowing constraining convex sets such as cones whose boundary manifolds are not smooth everywhere. Interestingly enough, this lemma may serve as a generalization of the Gaussian tail bound results in \cite{HasHu2003,Has2005} to settings where the measure concentrates in closed convex sets; note that these existing results, which are themselves generalizations of the classic Mills ratio bounds \cite{Sa1962}, exclusively focus on simple box/block sets.

While the proof of the rLDP upper bound is straightforward, establishing the lower bound involves carefully working with the local structure of the convex set around what turns out to be the dominating point. More precisely, the proof of the lower bound proceeds by establishing the corresponding result for a standard Gaussian random vector first. The key difficulty here lies in the fact that the constraining convex set need not be spatially symmetric about the dominating point, making it less obvious as to what type of geometric structure can be used to show the asymptotic concentration of the probability mass near the dominating point. Nonetheless, we demonstrate that the intersection of two normal cones is sufficient to show the asymptotic concentration; indeed, we believe this is the ``minimal" geometric structure necessary, and any other structure will require further regularity conditions on the convex set. Next, using a Cholesky decomposition of the covariance matrix, the isotropic analysis is transformed to the general setting. 

A critical issue in the multivariate setting is how one should define the multivariate extreme event. One natural path is to consider the ``at least one in the set" extreme event, which stipulates that at least one in the collection $\{X_1,X_2,\ldots,X_n\}$ enter the set $\sC$ (after appropriate scaling) --- see~\cite{DeKoMaRo2010} for instance in the context of continuous time stochastic processes. With the ``at least one in the set" definition, a number of technical issues pertaining to rLDPs are easily resolved;  furthermore, a straightforward utilization of indicator functions of convex sets~\cite{Ro2015} allows for generalization to the context of set inclusion. In Lemma~\ref{lem:convex} we prove an rLDP for the ``at least one in the set" extremum event using simple arguments. 

As can be seen in (\ref{blext}), however, we are interested in a more general  multivariate extremum definition than is allowed by the ``at least one in the set" definition. Observe that the ``at least one in the set" extremum event {\it requires} that at least one of the random vectors in the collection $\{X_1,X_2,\ldots,X_n\}$ enter the set $\sC$ after appropriate scaling. On the other hand, under our chosen definition of $M_n$, it is possible for the elements in the collection $\{X_1,X_2,\ldots,X_n\}$ to conspire such that {\it none} of the individual elements of the collection $\{X_1,X_2,\ldots,X_n\}$  enter the set $\sC$ but the block maximum $M_n$ does. This scenario is illustrated in Figure~\ref{fig:3}. In Lemma~\ref{lem:exp-eq} we demonstrate that under a specific asymptotic scaling, the measures corresponding to $M_n$ and the ``at least one in the set" extremum event are asymptotically exponentially equivalent in the sense of \cite[Def. 4.2.10]{DeZe2010}.  

\subsection{Related Literature}
The main result in this paper draws on insights from probability theory, particularly Gaussian ensembles, and convex optimization/geometry. Consequently, there are a number of connections to important and closely related research threads. First, as noted before, there is an explicit connection with EVT. The logarithmic asymptotics established here complement the uniform convergence results for EVT; see \cite[Chapter 4]{Vi1994} and \cite[Section I, Chapter C]{de1983}.

Next, there are clear connections with recent work on extremes of multidimensional Gaussian processes in \cite{DeKoMaRo2010,DeHaJiTa2015,KoMa2015,DeHaJiRo2018} and other related work, where logarithmic asymptotics are derived for the ``at least one in the set" extremum for Gaussian processes. We note, in particular, ~\cite{KoMa2015} where logarithmic asymptotics are derived for the ``at least one in set" extremum of a sequence of (non-iid) generally distributed random vectors. The authors present a general theory closely aligned with the restricted LDP method for univariate random variables introduced in~\cite{DuLeSu2003}, whereby the Gartner-Ellis condition need not be satisfied. Of course, our results are more restrictive in the sense that we only study iid Gaussian random vectors, but we also consider entry into general convex extreme value sets.

Our results are also closely related to the important series of papers by Hashorva and H{\"u}sler~\cite{HasHu2002,HasHu2003,Has2005} generalizing the classic Mills ratio Gaussian tail bound~\cite{Sa1962}. We observe that the quadratic program logarithmic asymptote derived in Lemma 1 is also implied by the tail bound dervied in~\cite{HasHu2003,Has2005} for the standard block cone. In~\cite{HasHu2003}, the authors derive exact asymptotics for integrals of Gaussian random vectors, and in particular focus on the ``at least one in the set" extremum for half-space extreme value sets. Our proof does not rely on the bound in \cite{HasHu2003,Has2005} and we establish our result for general closed convex sets and the block maximum. Further, while our results are established in the iid setting, they also provide a roadmap for establishing logarithmic asymptotics for stochastic processes more generally, extending recent work in \cite{DeKoMaRo2010} to a less restrictive definition of the extremum. 

The rest of the paper is organized as follows. In Section~3, we establish our main result in Theorem~\ref{thm:5} as a consequence of Lemma 1, Lemma 2 and Lemma 3. We then provide some additional insight in Proposition~\ref{prop:insight} into the location of the dominating point. In Section~4 we present a number of special cases, starting with the block set in Proposition~\ref{prop:boxes}, then polyhedral sets in Proposition~\ref{thm:4} and, finally, Gaussian mixture models in Proposition~\ref{prop:gmm}. We start by summarizing the notation and definitions used in Section~2.

\section{Notation and Definitions}
We assume that all vectors lie in the $d$-dimensional real space, $\bbR^d$. We will use the superscript to denote the components of a vector in $\bbR^d$, that is, $x = (x^1,x^2,\ldots,x^d) \in \bbR^d$. We will sometimes find it convenient to use the standard basis vectors $e_i \in \bbR^d, i = 1,2,\ldots,d$. We will sometimes use $\|x\|$ to refer to the Euclidean norm $\sqrt{\langle x, x \rangle}$ of the vector $x \in \bbR^d$. The boundary of any set $\sC$ is denoted by $\partial \sC$. We refer to the subdifferential of a function $g$ by $\partial g$; there should be no confusion from the context. All binary relations between vectors are defined componentwise. Next, we note the following definitions used throughout.

Our interest extends to settings where the
extrema lie in a closed (and hence measurable) convex subset of $\bbR^d$. A definition in this more
general setting depends on the inclusion event $\{X \in \sC\}$,
which is most naturally defined using 
\begin{definition}[Indicator Function]~\label{def:if}
  Let $\sC \in \bbR^d$ be an arbitrary subset. Then,
  \begin{align}
  \label{eq:17}
  \delta_{\mathcal C}(x) :=
  \begin{cases}
    0 &~\text{if}~ x \in \mathcal C\\
    -\infty &~\text{if}~x \not \in \mathcal C.
  \end{cases}
  \end{align}
\end{definition}

\input{extrema}

\begin{definition}[``At least one in the set" Extremal Event]
  Let $\{ X_n\}$ be an ensemble of $d$-dimensional 
  random vectors, then for any closed set $\sC$
  \begin{align}
    \label{eq:42}
    \left\{\tilde{X}_n \in \sC \right\} := \left\{\exists i \leq n:  \d_\sC(X_i) \geq 0 \right\}
  \end{align}
  is the ``at least one in the set" extremal event.
\end{definition}

\begin{definition}[Componentwise Extremal Event]
  The componentwise extremal event of an ensemble of $d$-dimensional random vectors $\{ X_n\}$ on a closed set $\sC$ is defined as
  \begin{align}
    \label{eq:34}
\left\{M_n \in \sC \right\} = \left\{ \d_\sC(M_n) \geq 0\right\},
  \end{align}
  where $M_n := (\max_{1\leq i\leq n} X_i^1, \ldots, \max_{1\leq i\leq n} X_i^d)$ is the componentwise maximum.
\end{definition}
Figure~\ref{fig:3} contrasts the two definitions of extrema, with the circles representing random vectors in the quadrant, and one of them lies in the set $\sC$ (shown as the hatched parabola). The componentwise maximum is the $\color{red}\star$, generated componentwise from the samples. Clearly, the two definitions of extrema need not coincide.

Next, given an ensemble of stationary random vectors $\{ X_n\} \subset \bbR^d$ and a closed set $\sC \subset \bbR^d_+$ (the positive orthant), we defined an \emph{atypical set} as follows. 

\begin{definition}[Atypical Set]
  If $\sC \not \ni \bbE[ X_n]$ for all $n \geq 1$, then $\sC$ is an {\it atypical set}.
\end{definition}

We will establish logarithmic asymptotics that correspond to a \textit{restricted} form of the large deviation principle (rLDP) akin to~\cite{DuLeSu2003,KoMa2015}. We emphasize that we do not establish a full-fledged LDP, since our results are exclusively ``one-sided," focused on deviations into the positive orthant. Unlike~\cite{DuLeSu2003,KoMa2015}, however, we require the scaling functions to be regularly varying at infinity.

\begin{definition}[Normal Cone, Hoop, Hoop-Segment, and Truncated Hoop-Cone Segment]\label{def:truhooconseg} The truncated hoop-cone segment is an object that will be used in the proof of Lemma~\ref{prop:convex}. It's precise description involves three much simpler objects, each of which we define below. See Figure~\ref{hoop-cones} for intuition in $\bbR^3$. \begin{enumerate} \item[(i)] A \emph{normal cone} $\mathcal{N}(z^*,y)$ having vertex $z^* \in \bbR^d$ and passing through the point $y \in \bbR^d, y \neq z^*$ is given by
$$\mathcal{N}(z^*,y) := \left\{z : ((z-z^*)^Te_d)^{-1} \left(\sum_{i=1}^{d-1} (z^Te_i)^2\right)^{\frac{1}{2}} \leq ((y-z^*)^Te_d)^{-1} \left(\sum_{i=1}^{d-1} (y^Te_i)^2\right)^{\frac{1}{2}} \right\}.$$ \item[(ii)] A \emph{hoop} $\mathcal{H}(z^*,z_0,w)$ having vertex $z^*$, offset $z_0$, and width $w$ is given by $$\mathcal{H}(z^*,z_0,w) := \bigg\{\mathcal{N}(z^*,z_0 + \frac{w}{2}) \setminus \mathcal{N}(z^*,z_0 - \frac{w}{2})\bigg\} \bigcap \bigg\{z \in \bbR^d: z^Te_d = z_0^Te_d\bigg\}.$$ \item[(iii)] A \emph{hoop-segment} $\mathcal{HS}(z^*,z_0,w,\ell)$ having vertex $z^*$, offset $z_0$, width $w$, and arc-length $\ell$ is given by $$\mathcal{HS}(z^*,z_0,w,\ell) := \mathcal{H}(z^*,z_0,w) \bigcap \bigg\{z \in \bbR^d: \frac{(z-\tilde{z}_0)^T(z_0 - \tilde{z}_0)}{\|z-\tilde{z}_0\| \|z_0 - \tilde{z}_0 \|} \leq \cos \left(\frac{\ell/2}{\|z_0 - \tilde{z}_0\| + \frac{w}{2}}\right)\bigg\}.$$ \item[(iv)] A \emph{truncated hoop-cone segment} having vertex $z^*$, offset $z_0$, width $w$, and arc-length $\ell$ is the convex hull $\mbox{conv}\left(z^*,\mathcal{HS}(z^*,z_0,w,\ell)\right)$ of $z^*$ and the hoop-segment $\mathcal{HS}(z^*,z_0,w,\ell)$. \end{enumerate} 

\end{definition}

%% file: extrema.tex
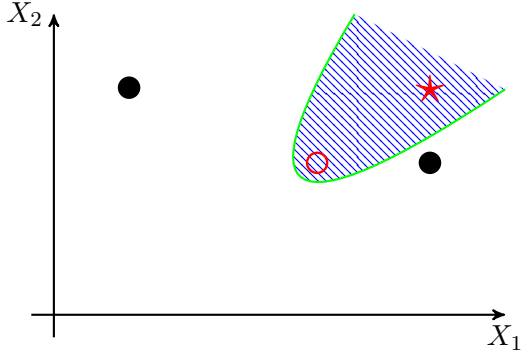
\begin{figure}[h]\centering
  \begin{tikzpicture}[
    thick, >=stealth', dot/.style =
      { draw, fill = none, circle, inner sep = 0pt, minimum size = 10pt
      }]
      \coordinate(0) at (0,0);
      \draw[->] (-0.3,0) -- (6,0) node[below] {$X_1$};
      \draw[->] (0,-0.3) -- (0,4) node[left] {$X_2$};   
      \fill[pattern=north west lines, pattern color=blue,dashed] (6,3)
      .. controls (3,1) and (2.5,1.5) ..(4,4);
      \draw[green] (6,3) .. controls (3,1) and (2.5,1.5) .. (4,4);
      \node[black,scale=2.0] at (1.0,3.0) {\textbullet};
      \node[black,scale=2.0] at (5.0,2.0) {\textbullet};
      \node[red,scale=2.0] at (5.0,3.0) {$\star$};
      \node[red,scale=2.0,] at (3.5,2.0) {$\circ$};
  \end{tikzpicture}
  \caption{Contrasting the extrema: three samples are shown as circles. The ``at least one in the set" extremum
  is the $\color{red} \huge\circ$, while the block extremum is the
  $\color{red} \huge\star$, generated as the componentwise maximum of
  the samples.}
  \label{fig:3}
\end{figure}

%% file: multivariate_v2.tex
\section{Main Results}
In this section, we characterize an rLDP of the type (\ref{rldpintro}), 
under the scaling regime implied by Assumption~\ref{assume:3}. The proposed rLDP  will follow from a sequence of lemmas establishing important
intermediate results.

Let $\{X_1,X_2,\ldots\}$ be an iid ensemble of $d$-dimensional
mean-zero Gaussian random vectors with positive-definite covariance matrix $\Sigma$, and let
$\sC \subset \bbR^d$ be a closed convex set. The following three assumptions will be used in the results that follow.
\begin{assumption}\label{assume:1} 
  The set $\sC$ is atypical, that is, it does not contain the origin. Furthermore,  \begin{enumerate} \item the set $\sC$ has at least one \emph{extreme} point, that is, a point $x \in \sC$ such that $\nexists \, x_1, x_2 \in \sC$ and $\alpha \in (0,1)$ such that $x = \alpha x_1 + (1- \alpha)x_2$; \item the set $\sC$ has an interior, that is, $\exists \, x_0 \in \sC$ and $\epsilon > 0$ such that the neighborhood $\sB(x_0,\epsilon) = \{x: \|x - x_0\| \leq \epsilon\} \subset \sC$.\end{enumerate}
\end{assumption}

\begin{assumption}\label{assume:3} The scaling matrix $\mathbf A_n := \textsc{diag}(a_{n,1},\ldots, a_{n,d})$ is a diagonal matrix that satisfies
    \begin{align}
      \frac{\mathbf A_n}{\| \mathbf A_n \|} &\to \mathbf A :=
      \textsc{diag}(a_1,\ldots,a_d),\\
      \frac{\| \mathbf A_n \|^2_2}{2\log n} &\to 1,~\text{as}~n\to\infty,
    \end{align}
    where $a_i > 0$ for all $i = 1,\ldots, d$ .
  \end{assumption}
 The sequence of scaling matrices $\mathbf A_n$ \textit{compresses} or shrinks the level sets of the multivariate Gaussian density function, or equivalently, it up-scales each dimension of a vector $\epsilon \in \partial \sC$ (roughly) to order
  $O(\sqrt{\log n})$. 
\begin{assumption}\label{assume:3c} The limit matrix $\mathbf
  A$ in Assumption~\ref{assume:3} is such that for any $\e \in \partial \sC$, $$\frac{1}{2}\langle  \mathbf A\mathbf \e,  \Sigma^{-1} \mathbf A\e
 \rangle  > 1.$$
\end{assumption}

Our first intermediate lemma establishes a limit result for the
log-likelihood that a scaled Gaussian random vector sample lies in a convex set $\sC$.

\begin{lemma}~\label{prop:convex}
	Let Assumption~\ref{assume:1} hold. Also, let $a_n$ be a sequence such that $a_n \to \infty$ as $n \to
	\infty$. Then,
	\begin{align}
	\label{eq:24}
	\lim_{n\to\infty} \frac{1}{a^2_n} \log \bbP\left( \d_\sC\left( \frac{X}{a_n} \right) \geq
	0\right) = -\frac{1}{2} {\left\langle x^* , \Sigma^{-1} x^* \right \rangle},
	\end{align}
	where $x^*= \underset{x \in \mathcal {C}}{\arg \inf}{\left\langle x , \Sigma^{-1} x \right \rangle}$.\\
\end{lemma}

\input{lbconvex.tex}
Note that $x^*$ exists and is unique since the objective $\left\langle x , \Sigma^{-1} x \right \rangle$ is a strictly convex quadratic function with unconstrained minimum at the origin, and the closed, convex constraint
set $\sC$ does not contain the origin. Lemma~\ref{prop:convex} can be interpreted as extending existing Gaussian tail bound results of the sort $\bbP(X_n >
t_n)$ as $t_n \to \infty$~\cite{HasHu2003,Has2005} to the context of general convex atypical sets. While only logarithmic asympotics are considered in Lemma~\ref{prop:convex} due to the nature of results sought in the current paper, the proof of Lemma~\ref{prop:convex} also seems to provide a path toward proving exact asymptotics. 

The proof of the Lemma~\ref{prop:convex} proceeds by constructing upper and lower bounds that are shown to coincide for the standard multivariate Gaussian case, obtained after an appropriate transformation. The upper
bound follows from a (standard) Gaussian concentration inequality. Most of the complications arise in establishing a tight lower bound. The
key step in establishing the lower bound is to strategically use the
subdifferentials of the boundary of the convex set
to show that, in the asymptotic scale, the measure of the event
$\left\{\d_\sC\left( \frac{X}{a_n} \right) \geq 0 \right\}$ concentrates at
the intersection of the boundary and one of the axes. The inverse transform then projects the limit
result back to the original coordinate system. Figure~\ref{fig:pareto}
demonstrates the construction for a case where the boundary $\partial \sC$ is smooth in the neighborhood of the point $z^*$, representing
the closest point on the boundary to the origin. The corresponding
construction for sets $\sC$ having non-smooth boundaries is asymmetric
about the vertical axis, causing further technical complication that
is resolved through a symmetrizing operation. Notice that the right-hand side of Lemma~\ref{prop:convex} includes only the point $z^*$ and the covariance matrix of $X$; particularly, Lemma~\ref{prop:convex} is not fine enough to capture the effect of the curvature of the set $\sC$ around the point $z^*$. 

Next, we prove that the ``at least one in the set" event $\max_{1 \leq i\leq n}\d_\sC\left( \mathbf A_n^{-1} X_i
\right) \geq 0$ satisfies an asymptotic of the type (\ref{rldpintro}).
\begin{lemma}~\label{lem:convex}
	Let Assumption~\ref{assume:1}, Assumption~\ref{assume:3}, and Assumption~\ref{assume:3c} hold. Let the scaling matrix $\mathbf A_n$ and the limit matrix $\mathbf A$ be as defined in Assumption~\ref{assume:3} and Assumption \ref{assume:3c}, and let $x^*= \underset{x \in \mathcal {C}}{\arg \inf}{\left\langle x , \Sigma^{-1} x \right \rangle}$. Then,
	\begin{align}
	\label{eq:40}
	\lim_{n\to\infty} \frac{1}{\|\mathbf A_n \|^2_2} \log \bbP
	\left(\max_{1 \leq i \leq n} \d_\sC\left(\mathbf A_n^{-1}
	X_i\right) \geq 0 \right) = \frac{1}{2} -\frac{1}{2} {\left\langle x^* , \mathbf A^{-1} \Sigma^{-1} \mathbf A x^* \right \rangle},
	\end{align}
	where $\mathbf A$ is the limit matrix in Assumption~\ref{assume:3}.
      \end{lemma} 

The proof of the upper bound in Lemma~\ref{lem:convex} follows
directly from the iid assumption.  The lower bound uses
Lemma~\ref{prop:convex} to demonstrate that no mass escapes at infinity
(thereby trivializing the likelihood), provided that
$\frac{1}{2}\langle x^*, \mathbf A^{-1} \Sigma^{-1} \mathbf A^{-1} x^*
\rangle > 1$. The latter condition is critical, demonstrating a
``collusion" between the geometry of the set $\sC$ and the level sets
of the density function. In particular, the set $\sC$ needs to be
sufficiently far from the origin for the asymptotic to hold. This is in
contrast with Lemma~\ref{prop:convex}, where no such restriction is imposed.
      
To extend the result in Lemma~\ref{lem:convex} to the componentwise
maximum $M_n$ that is of interest, we observe that the random variables $\d_\sC\left( \mathbf A_n^{-1} M_n \right)$ and $\max_{1 \leq
    i\leq n} \d_\sC \left( \mathbf A_n^{-1} X_i\right)$ are exponentially equivalent
under the scaling in Assumption~\ref{assume:3}. 


\begin{lemma}~\label{lem:exp-eq}
 	Let Assumption~\ref{assume:1}, Assumption~\ref{assume:3}, and Assumption~\ref{assume:3c} hold. Let the scaling matrix $\mathbf A_n$ and the limit matrix $\mathbf A$ be as defined in Assumption~\ref{assume:3} and Assumption \ref{assume:3c}, and let $x^*= \underset{x \in \mathcal {C}}{\arg \inf}{\left\langle x , \Sigma^{-1} x \right \rangle}$. Then, for any $\g > 0$
  \begin{align}
    \label{eq:41}
    \limsup_{n\to\infty} \frac{1}{\|\mathbf A_n\|^2_2} \log \bbP \left(
    \left|\d_\sC\left( \mathbf A_n^{-1} M_n \right) - \max_{1 \leq
    i\leq n} \d_\sC \left( \mathbf A_n^{-1} X_i\right) \right|  > \gamma
    \right) = -\infty.
  \end{align}
\end{lemma}

The intuition behind the proof of Lemma~\ref{lem:exp-eq} is the
following: observe that the individual samples can ``conspire" to
create a componentwise extremum that is included in the set $\sC$, {\it
even if none of them are}. However, under the large deviation scaling,
we argue that the only way that the componentwise extremum can lie in
the set, is if it was actually observed in the sample. The proof then
identifies a box that is anchored at the orgin and diagonally across
at $x^*$ such that all the sample values are scaled down to lie
precisely in this box. Once again, it turns out this is only feasible if
$\frac{1}{2}\langle x^*, \mathbf A^{-1} \Sigma^{-1} \mathbf A x^*
\rangle > 1$.

Finally, it follows that the extreme value statistic $M_n$ satisfies an rLDP with rate $J(x^*) = \frac{1}{2} -\frac{1}{2} {\left\langle x^* , \mathbf A^{-1} \Sigma^{-1} \mathbf A x^* \right \rangle}$. 

\begin{theorem}\label{thm:5}
	With the scaling matrix $\mathbf A_n$ and the limit matrix $\mathbf A$ as defined in Assumption~\ref{assume:3} and Assumption \ref{assume:3c}, and $x^*= \underset{x \in \mathcal {C}}{\arg \inf}{\left\langle x , \Sigma^{-1} x \right \rangle}$, the componentwise maximum $M_n$ defined in (\ref{blext}) satisfies an rLDP with rate $J(x^*) := \frac{1}{2} -\frac{1}{2} {\left\langle x^* , \mathbf A^{-1} \Sigma^{-1} \mathbf A x^* \right \rangle}$ and speed $v(a) = a^2$, that is,
  \begin{align}
    \label{eq:39}
    \lim_{n\to\infty} \frac{1}{\|\mathbf A_n\|^2_2} \log \bbP
    \left(\d_\sC\left(\mathbf A_n^{-1} M_n \right) \geq 0 \right)
    = \frac{1}{2} -\frac{1}{2} {\left\langle x^* , \mathbf A^{-1} \Sigma^{-1} \mathbf A x^* \right \rangle}.
  \end{align}

\end{theorem}

\begin{proof}
  The proof is a direct consequence of Lemma~\ref{prop:convex}, Lemma~\ref{lem:convex} and Lemma~\ref{lem:exp-eq}.
\end{proof}

Some comments are in order for this result. The constrained quadratic program is determined by
the level sets of the Gaussian distribution, the geometry of the
constraint set $\sC$ and the scaling matrix $\mathbf A$. Intuitively, one should expect an interplay
between the closest point to the boundary of the set $\sC$ and the primary
eigenvector of the Gaussian covariance matrix to determine the most
likely way in which the contours of the (multivariate) Gaussian
density function first touch the set $\sC$. The asymptotic result above
shows that there is a ``collusion" between these quantities, resulting
in a quadratic program that determines the location of the {\it dominating point} $x^*$ on the boundary of the set. Intuitively, the result in Theorem~\ref{thm:5} implies that the dominating point $x^*$ represents the point at which all of the measure concentrates in the  rLDP limit, analogous to \cite{Ne1983} in the random walk case. This insight
could play a crucial role in the design of sampling algorithms for rare
events associated with the excursions of iid Gaussian
processes.

What is the nature of the dominating point $x^* = \underset{x \in \mathcal {C}}{\arg \inf}{\left\langle x , \Sigma^{-1} x \right \rangle}$? Since the optimization problem $\underset{x \in \mathcal {C}}{\arg \inf}{\left\langle x , \Sigma^{-1} x \right \rangle}$ involves minimizing a convex quadratic objective over a convex set $\sC$, the necessary and sufficient condition for a minimum at $x^*$ is that the directional derivative along any feasible direction at $x^*$ is positive, that is, $x = x^*$ is such that 
\begin{equation} \label{fritzkkt} d^T \Sigma^{-1} {\mathbf A} x^* > 0, \mbox{ for all } \, d \in \mathcal{F}(x^*), \nonumber\end{equation} where the feasible set of directions $\mathcal{F}(x^*)$ at $x^*$ is given by $$\mathcal{F}(x^*) : = \bigg\{d: d \neq 0 \mbox{ and } x^* + \lambda d \in \sC \mbox{ for all } \lambda \in (0,\delta) \mbox{ for some } \delta >0\bigg\}.$$ Identifying the point $x^*$ turns out to be especially easy when the set $\sC$ is expressed using functional inequalities and is the subject of convex optimization techniques~\cite{2004boyvan}.

An illuminating special case occurs when the eigen vectors of the covariance matrix $\Sigma$ are oriented such that $\Sigma^{-1} z > 0$ for all $z \in \sC$. In such a case, the point $x^*$ turns out to be the point in the set $\sC$ that is closest to the origin in Euclidean norm. 

\begin{proposition} \label{prop:insight}
  If $\Sigma^{-1} z > 0$ for all $z \in \sC$, then $$x^* := \arg\inf_{x
  \in \sC} \frac{1}{2}\langle x, \mathbf A^{-1} \Sigma^{-1} \mathbf A
x \rangle = \arg\inf_{x
  \in \sC} \frac{1}{2}\langle x, 
x \rangle.$$
\end{proposition}

\begin{proof}
   Without loss of generality, assume that $\mathbf A = \mathbf I$, the identity
  matrix.

  Consider the following constrained optimization problems:
  \begin{align}\label{cop:P}
    \tag{$P$}
    \min &~\frac{1}{2} x^T x\\
    \notag
    s.t. & ~x \in \sC,
  \end{align}
  where $\sC$ is convex, and
  \begin{align}\label{cop:tilde-P}
    \tag{$\tilde{{P}}$}
    \min &~\frac{1}{2} x^T \Sigma^{-1}x\\
    \notag
    s.t. & ~x \in \sC;
  \end{align}
  recall that $\Sigma^{-1}$ is a positive definite matrix.

  Now, suppose that $x^*$ is the unique solution of~\eqref{cop:P}, and
  assume without loss of generality that $x^* > 0$. Note that this is
  precisely the closest point on the set $\sC$ to the origin (in
  Euclidean distance). We will have shown
  that $x^*$ is also the unique minimizer of~\eqref{cop:tilde-P}
  if we can prove that the directional derivative along any feasible
  direction $x$ (at $x^*$) is strictly positive. In otherwords, $x^*$
  will be the solution to~\eqref{cop:tilde-P} if
  \begin{align}
    \left( \nabla \tilde f(x)\bigg|_{x = x^*}\right)^T x > 0,~\tilde
    f(x) = \frac{1}{2} x^T \Sigma^{-1} x
  \end{align}
  for any direction $x$ that is feasible at $x^*$. That is, we require
  that
  \begin{align*}
    (x^*)^T \Sigma^{-1} x> 0
  \end{align*}
  for each feasible direction $x$ at $x^*$. Since $\Sigma^{-1} x > 0$
  by assumption, for all feasible directions $x$, we conclude that
  $x^*$ is also a minimizer of~\eqref{cop:tilde-P}.
\end{proof}

If the support function of the set $\sC$ is strictly convex, then
$x^*$ is necessarily an extreme point, indicating that the dominating
point is also an extreme point of the set $\sC$. Importantly, and potentially contrary to one's intuition, Proposition \ref{prop:insight} implies that the dominating point $x^* := \arg\inf_{x
  \in \sC} \frac{1}{2}\langle x, \mathbf A^{-1} \Sigma^{-1} \mathbf A
x \rangle$ may not in general coincide with the point that is closest (in Euclidean norm) to the origin. It is in this sense that the structure of the set $\sC$ and the nature of the covariance matrix are in collusion to determine the location of the unique dominating point.

We now present the proofs of the preceding lemmas.

\begin{proof}[Lemma~\ref{prop:convex}]
	
	Since $\Sigma^{-1}$ is a symmetric positive definite matrix,
        there exists a unique $d \times d$ matrix $L$, such that  $L^TL= \Sigma^{-1}$. Observe that 
	\begin{align}\label{reform}
	\bbP\left( \d_\sC\left( \frac{X}{a_n} \right) \geq 0\right) = \bbP\left( \d_{L\sC}\left( \frac{LX}{a_n} \right) \geq 0\right) ,
	\end{align} where the set $L\sC:= \{Ly: y \in \sC \}$ is easily seen to satisfy Assumption~\ref{assume:1}, and $LX$ is a d-dimensional
        standard normal. Hence, considering (\ref{reform}), we will prove the lemma for a standard normal random vector $Z=LX$ on the convex atypical set $L\sC$, giving the dominating point $z^* = \underset{z \in \mathcal {LC}}{\arg \inf}{\left\langle z , z \right
          \rangle}$. 
          
Using the fact that ${z^*}^Tz=\|z^*\|^2$ is a tangent to the ball of radius $\|z^*\|^2$ centered at the origin, and also to the convex set $L\mathcal{C}$ at $z^*$, we see that ${z^*}^Tz = \|z^*\|^2$ is a supporting hyperplane to the set $L\mathcal{C}$ at $z^*$; and since $L\sC$ is atypical, ${z^*}^Tz \geq \|z^*\|^2$ for all $z \in \mathcal L\sC$. It then follows that
	\begin{align} \label{udbsupp}
	\bbP\left( \d_{L\sC} \left( \frac{Z}{a_n} \right) \geq 0\right) & \leq \bbP\left( {{z^*}^TZ}   \geq {a_n} \|z^*\|^2 \right).
	\end{align} Using the Mill's ratio~\cite{Sa1962} for the univariate Gaussian random variable ${z^*}^TZ$ in (\ref{udbsupp}), we get
	\begin{align*}
	\log \bbP\left( \d_{L\sC} \left( \frac{Z}{a_n} \right) \geq 0\right) & \leq  -\frac{1}{2}  a_n^2 \|z^*\|^2   -\log a_n \|z^*\|^2  - \frac{1}{2} \log 2\pi,
	\end{align*} and it follows immediately that
	\begin{align}
	\label{eq:c1}
	\limsup_{n \to \infty} \frac{1}{a_n^2} \log \bbP\left( \d_{L\sC} \left( \frac{Z}{a_n} \right) \geq 0\right) & \leq  -\frac{1}{2} \|z^*\|^2,
	\end{align} thus proving the upper bound.
	\input{diffcone}
	
	Let us now prove the lower bound. First observe that since the density of $Z$ is spherically symmetric, we can assume for convenience, and without loss of generality, that the dominating point $z^*$ has components 1 through $d-1$ set to 0, that is, ${z^*}^T e_j = 0$ for $j=1,2, \ldots, d-1$ and ${z^*}^T e_d = \|z^*\|$. (This can be achieved by multiplying $Z$ by an appropriate rotation matrix.) The supporting hyperplane to $L\sC$ at $z^*$ is then perpendicular to the $d$-th axis. Since we have assumed that the set $L\sC$ has an interior, there exists a point $z_0 \neq z^*$ and $\delta > 0$ such that the Euclidean ball $$B(z_0, \delta) = \bigg\{z: \|z-z_0\| \leq \delta \bigg\} \subset L\sC.$$ Let $\tilde{z}_0$ be the projection of $z_0$ on the $d$-th axis, that is, $\tilde{z}_0^Te_i = 0$ for $i=1,2,\ldots,d-1$ and $\tilde{z}_0^Te_d = \|z_0\|$. Now, since $L\sC$ is convex, we can choose $\epsilon>0$ such that the trucated hoop-cone segment $\mathcal{TC}(z^*, z_0, \epsilon/2, \epsilon/4)$ having vertex $z^*$, offset $z_0$, width $\epsilon/2$ and arc-length $\epsilon/4$ satisfies \begin{equation}\label{canonicalobj} \mathcal{TC}(z^*, z_0, \epsilon/2, \epsilon/4) \subset L\sC. \end{equation} (See Definition~\ref{def:truhooconseg} for a precise definition of the truncated hoop-cone segment.) Furthermore, identify offsets $z_1, z_2, \ldots, z_n$ so that the truncated hoop-cone segments $\mathcal{TC}(z^*, z_i, \epsilon/2, \epsilon/4), i = 0,1,2,\ldots,n$ are all symmetrically located about $\tilde{z}_0$. (See Figure~\ref{hoop-cones} for the truncated hoop-cone segments $\mathcal{TC}(z^*, z_i, \epsilon/2, \epsilon/4), i=0,1,,\ldots,n$.) Due to  (\ref{canonicalobj}), and since the construction is such that the truncated hoop-cone segments are symmetric about the $d$-th axis, we see that \begin{equation}\label{lblb} \bbP\left(\frac{Z}{a_n} \in  L\sC\right) \geq  \bbP\left(\frac{Z}{a_n} \in  \mathcal{TC}(z^*, z_0, \epsilon/2, \epsilon/4)\right) = \frac{1}{n+1}\bbP\left(\frac{Z}{a_n} \in  \bigcup_{i=0}^n \mathcal{TC}(z^*, z_i, \epsilon/2, \epsilon/4)\right).
	\end{equation} We will now characterize $\bbP\left(\frac{Z}{a_n} \in  \bigcup_{i=0}^n \mathcal{TC}(z^*, z_i, \epsilon/2, \epsilon/4)\right)$ appearing in (\ref{lblb}) to establish the needed lower bound. For this, notice that\begin{align}\label{hoopconeunion} \MoveEqLeft \bigcup_{i=0}^n \mathcal{TC}(z^*, z_i, \epsilon/2, \epsilon/4) = & \nonumber \\ & \mathcal{N}\left(z^*,\tilde{z}_0 + (1+\frac{\epsilon}{4})(z_0 - \tilde{z}_0)\right) \setminus \mathcal{N}\left(z^*,\tilde{z}_0 + (1-\frac{\epsilon}{4})(z_0 - \tilde{z}_0)\right) \setminus \left\{ z: z^Te_d > \|\tilde{z}_0^*\|\right\}. 
	\end{align} where $\mathcal{N}(z^*,y)$ is the Euclidean normal cone with vertex $z^*$ and having boundary passing through the point $y \in \bbR^d$. (See Definition~\ref{def:truhooconseg} for a precise definition of the normal cone $\mathcal{N}(z^*,y)$.) 
	
	Defining the slopes $\kappa_{\mbox{\scriptsize outer}} = \frac{\|\tilde{z}_0 - z^*\|}{(1+\epsilon/4)\|z_0 - \tilde{z}_0\|}, \kappa_{\mbox{\scriptsize inner}} = \frac{\|\tilde{z}_0 - z^*\|}{(1-\epsilon/4)\|z_0 - \tilde{z}_0\|}$ of the outer and inner cones, respectively, we can write the equations of the truncated outer cone $\bar{\sC}_{\mbox{\scriptsize outer}}$ and the truncated inner cone $\bar{\sC}_{\mbox{\scriptsize inner}}$ as follows.   \begin{equation}\label{cart1} \bar{\sC}_{\mbox{\scriptsize outer}} := \bigg\{(x,y) \in \bbR^{d-1} \times \bbR: \|z^*\| + \kappa_{\mbox{\scriptsize outer}} \|x\| \leq y  \leq  \|\tilde{z}_0\| \bigg \}
	\end{equation} and \begin{equation}\label{cart2} \bar{\sC}_{\mbox{\scriptsize inner}} := \bigg\{(x,y) \in \bbR^{d-1} \times \bbR: \|z^*\| + \kappa_{\mbox{\scriptsize inner}} \|x\| \leq y  \leq  \|\tilde{z}_0\| \bigg\}.\end{equation} Combining (\ref{hoopconeunion}), (\ref{cart1}), and (\ref{cart2}), we can write \begin{align}\label{hoopconeunionexp2} \bigcup_{i=0}^n \mathcal{TC}(z^*, z_i, \epsilon/2, \epsilon/4) = \bar{\sC}_{\mbox{\scriptsize outer}} \setminus \bar{\sC}_{\mbox{\scriptsize inner}}. 
	\end{align} Denoting $Z= ( Z_1, Z_2, Z_3 \ldots ,Z_d )$, $Z_{-d}= ( Z_1, Z_2, Z_3 \ldots,Z_{d-1}  )$, and $t^* := \frac{\|\tilde{z}_0\| - \|z^*\|}{\kappa_2}$, the representation in (\ref{hoopconeunionexp2}) along with (\ref{lblb}) implies
	\begin{align} 
	\nonumber
	\bbP\left( \d_{L\sC} \left( \frac{Z}{a_n} \right) \geq 0\right) \geq & \frac{1}{n+1}\bbP\left( \d_{\bar{\sC}_{\mbox{\scriptsize outer}} \setminus \bar{\sC}_{\mbox{\scriptsize inner}} }\left( \frac{Z}{a_n} \right) \geq 0\right),\\ 
	\geq & \frac{1}{n+1}\int_0^{t^* } \bbP \left( a_n \|z^*\| +\kappa_{\mbox{\scriptsize outer}} a_n t \leq Z_d \leq a_n \|z^*\| + \kappa_{\mbox{\scriptsize inner}}a_n t \right) \bbP_{\chi_{d-1}}(dt),
	\label{lbsetup} 
	\end{align}
	where $\chi_{d-1}$ is a Chi random variable with $(d-1)$ degrees of freedom. Applying the Mill's ratio~\cite{Sa1962} to the integrand in (\ref{lbsetup}), we get \begin{align} \label{lbsetup2}
	\nonumber
	\MoveEqLeft \bbP\left( \d_{L\sC} \left( \frac{Z}{a_n} \right) \geq 0\right) & \\ & \geq   \frac{1}{\sqrt{2\pi}} \frac{(n+1)^{-1}}{a_n(\|z^*\| + \kappa_{\mbox{\scriptsize inner}} t^*)} \int_0^{t^* } \left(\exp\{-\frac{1}{2}a_n^2(\|z^*\| + \kappa_{\mbox{\scriptsize outer}}t)^2\} -  \exp\{-\frac{1}{2}a_n^2(\|z^*\| + \kappa_{\mbox{\scriptsize inner}}t)^2\} \right) \bbP_{\chi_{d-1}}(dt) \nonumber \\
	& = \frac{1}{\sqrt{2\pi}} \frac{(n+1)^{-1}}{a_n(\|z^*\| + \kappa_{\mbox{\scriptsize inner}} t^*)} \exp\bigg\{-\frac{1}{2}a_n^2\|z^*\|^2\bigg\} I_n,
	\end{align} where the asymptotic \begin{align} \label{remconst} I_n & := \int_0^{t^* } \left(\exp\bigg\{-\frac{1}{2}a_n^2(\kappa_{\mbox{\scriptsize outer}}^2t^2 + 2\kappa_{\mbox{\scriptsize outer}}t\|z^*\| )\bigg\} -  \exp\bigg\{-\frac{1}{2}a_n^2(\kappa_{\mbox{\scriptsize inner}}^2t^2 + 2\kappa_{\mbox{\scriptsize inner}}t\|z^*\|)\bigg\} \right) \bbP_{\chi_{d-1}}(dt) \nonumber \\ & \sim \left(\frac{2(d-2)!}{2^{(d-1)/2}\Gamma\left(\frac{d-1}{2}\right)} \right) \left(\frac{1}{\kappa_{\mbox{\scriptsize outer}}^{d-1}} - \frac{1}{\kappa_{\mbox{\scriptsize inner}}^{d-1}}  \right) \left(\frac{1}{\|z^*\|a_n^{2}}\right)^{d-1} \end{align} can be seen after some calculation. From (\ref{remconst}) and (\ref{lbsetup2}), we see that 
\begin{align}
	\label{lowerbd}
	\liminf_{n \to \infty} \frac{1}{a_n^2} \log \bbP\left( \d_{L\sC} \left( \frac{Z}{a_n} \right) \geq 0\right) & \geq  -\frac{1}{2} \|z^*\|^2_2,
	\end{align} proving the lower bound. 
\end{proof}

\begin{proof}[Lemma~\ref{lem:convex}]
	Observe that independence implies 
	\begin{align*}
	\bbP\left( \max_{1\leq i\leq n} \d_\sC\left( \mathbf A_n^{-1} X_i
	\right) \geq 0\right) &= 1 - \bbP\left( \d_\sC\left( \mathbf
                                A_n^{-1} X_1\right) \leq 0 \right)^n\\
          &= \bbP\left( \d_\sC(\mathbf A_n^{-1} X_1 \geq 0)\right)
            \sum_{i=0}^{n-1}\bbP\left( \d_{\sC}(\mathbf
            A_n^{-1} X_1) < 0 \right)^{i}.
	\end{align*}
	The obvious probability upper bound $\sum_{i=0}^{n-1}\bbP\left( \d_{\sC}(\mathbf
            A_n^{-1} X_1) < 0 \right)^{i} \leq n$ together with
        Assumption~\ref{assume:3} and Lemma~\ref{prop:convex} automatically imply that
	\begin{align}
	\label{eq:38}
	\limsup_{n\to\infty} \frac{1}{\|\mathbf A_n\|_2^2} \log \bbP\left( \max_{1\leq i\leq n} \d_\sC\left( \mathbf A_n^{-1} X_i\right) \geq 0\right)  \leq \frac{1}{2} -\frac{1}{2} {\left\langle x^* , \mathbf A^{-1} \Sigma^{-1} \mathbf A x^* \right\rangle};
	\end{align}
        (Note that while the sequence $a_n$ in Lemma~\ref{prop:convex}
        is arbitrary, minor algebra is required to extend that result
        to a limit for $\mathbf A_n^{-1} X_1$.)

        Next, for the lower bound we have
        \begin{align*}
          \bbP\left( \max_{1\leq i\leq n} \d_\sC\left( \mathbf A_n^{-1} X_i
	\right) \geq 0\right) &\geq \bbP\left( \d_\sC(\mathbf A_n^{-1}
                                X_1 \geq 0)\right) \times n \bbP\left( \d_{\sC}(\mathbf
            A_n^{-1} X_1) < 0 \right)^{n-1}.
        \end{align*}
        Observe that
        \begin{align*}
          (n-1)\log \bbP\left(\d_\sC(\mathbf A_n^{-1} X_1) < 0\right) = (n-1)
          \log \left(1-\bbP\left(\d_\sC(\mathbf A_n^{-1} X_1) \geq 0\right) \right) \in o(1).
        \end{align*}
        To see this, Taylor expand the $\log$ term to obtain
        \begin{align*}
           \log \left(1-\bbP\left(\d_\sC(\mathbf A_n^{-1} X_1) \geq
          0\right) \right) &= - \bbP\left(\d_\sC(\mathbf A_n^{-1} X_1) \geq
          0\right) - \frac{1}{2} \bbP\left(\d_\sC(\mathbf A_n^{-1} X_1) \geq
          0\right)^2 - \cdots\\
          &\asymp - \exp\left( - \frac{\|\mathbf A_n\|_2^2}{2} \a
            \right) - \frac{1}{2} \exp\left( -\|\mathbf A_n\|_2^2
            \a\right) - \cdots\\
          &\sim - \exp\left(-\a \log n \right) - \frac{1}{2}\exp\left(
            - 2 \a \log n\right) - \cdots\\
          &= - \frac{1}{n^\a} - \frac{1}{2 n^{2\a}} - \cdots,
        \end{align*}
        where $\a = \frac{1}{2} \langle x^*, \mathbf A^{-1} \Sigma^{-1}
        \mathbf A x^* \rangle$ and $\asymp$ represents the order of
        magnitude estimate; that is $f(s) \asymp g(s)~s \in S$ is
        equivalent to $f(s) \in O(g(s))$ and $g(s) \in O(f(s))$. The
        order of magnitide approximation in the
        second display follows from Lemma~\ref{prop:convex}, while
        Assumption~\ref{assume:3} implies the third
        display. Assumption~\ref{assume:3c} implies that $\a >
          1$, so that
          \begin{align}\label{eq:lb-approx}
            (n-1) \log \bbP\left(\d_\sC(\mathbf A_n^{-1} X_1) <
            0\right) \asymp - \frac{(n-1)}{n^\a} - \frac{n-1}{2
            n^{2\a}} - \cdots \in o(1).
          \end{align}

          Finally, observe that
          \begin{align*}
            	\log \bbP\left( \max_{1\leq i\leq n} \d_\sC\left( \mathbf
            A_n^{-1} X_i\right) \geq 0\right) \geq \log &\bbP\left(
            \d_\sC(\mathbf A_n^{-1} X_1) \geq 0\right)\\&\qquad + (n-1) \log \bbP\left( \d_{\sC}(\mathbf
            A_n^{-1} X_1) < 0 \right) + \log n.
          \end{align*}
          Lemma~\ref{prop:convex},~\eqref{eq:lb-approx} and
          Assumption~\ref{assume:3} together imply that
          \begin{align*}
            \liminf_{n\to\infty}\frac{1}{\|\mathbf A_n\|_2^2} \log \bbP\left( \max_{1\leq i\leq n} \d_\sC\left( \mathbf
            A_n^{-1} X_i\right) \geq 0\right) \geq \frac{1}{2} - \frac{1}{2} {\left\langle x^* , \mathbf A^{-1} \Sigma^{-1} \mathbf A x^* \right \rangle}.
          \end{align*}
\end{proof}

\begin{proof}[Lemma~\ref{lem:exp-eq}]
  Observe that
  \begin{align*}
    \max_{1 \leq i \leq n} \left\{\d_\sC\left(\mathbf A_n^{-1} X_i
    \right) - \d_\sC\left(\mathbf A_n^{-1} M_n\right) \right \} > \g
  \end{align*}
  is impossible by definition, since this would entail $M_n$ being outside
  $\sC$, while the sample values $X_1,\ldots$ $, X_n \in \sC$. Thus, the only case to consider is
  where
  \begin{align*}
    \max_{1 \leq i \leq n} \left\{\d_\sC\left(\mathbf A_n^{-1} X_i
    \right) - \d_\sC\left(\mathbf A_n^{-1} M_n\right) \right \} < -\g.
  \end{align*}
This can occur if and only if $X_1,\ldots,X_n \not \in \sC$ are such that $M_n
\in \sC$. That is, the sample points `conspire' to ensure that the
componentwise maximum is in the set $\sC$. That is,
\begin{align*}
  \left\{ \max_{1 \leq i \leq n} \left\{\d_\sC\left(\mathbf A_n^{-1} X_i
    \right) - \d_\sC\left(\mathbf A_n^{-1} M_n\right) \right \} <
  -\g\right\} = \left \{ \max_{1 \leq i\leq n} \d_\sC\left( \mathbf
  A_n^{-1} X_i\right) = -\infty, \d_\sC\left(\mathbf A_n^{-1} \bar
  X_n\right) = 0 \right\}.
\end{align*}
Now, let $B := \{x \in \bbR^d_+ : x^j \leq x^{*,j}, ~j=1,\ldots,d\}$,
where $x^* := \arg \inf_{x \in \sC} \frac{1}{2} \langle x, \mathbf
A^{-1}\Sigma^{-1} \mathbf A x \rangle$. For the moment, suppose we establish that
\begin{align}
  \label{eq:43}
  \bbP \left( \min_{1\leq i\leq n} \d_B\left( \mathbf A_n^{-1} X_i\right) = -\infty ~\bigg| \max_{1 \leq i \leq n} \d_\sC\left(\mathbf A_n^{-1} X_i
    \right) = -\infty\right) \to 0 ~\text{as}~n\to\infty;
\end{align}
that is if the sample values are outside the set $\sC$, then with high
probability they will cluster near the origin, in the set $B$. As a
consequence, it follows that
  \begin{align}
  \bbP\left( \d_B\left(\mathbf A_n^{-1} M_n \right) = -\infty ~\bigg| \max_{1 \leq i \leq n} \d_\sC\left(\mathbf A_n^{-1} X_i
    \right) = -\infty \right) \to 0~\text{as}~n\to\infty.\label{eq:44}
  \end{align}
  Observe that
  \begin{align*}
  &\bbP \left( \max_{1 \leq i\leq n} \d_\sC\left( \mathbf
  A_n^{-1} X_i\right) = -\infty, \d_\sC\left(\mathbf A_n^{-1} \bar
  X_n\right) = 0\right)\\ &= \bbP \left( \min_{1\leq i\leq n}\d_B\left(\mathbf A_n^{-1} X_i \right) = 0,~ \max_{1 \leq i\leq n} \d_\sC\left( \mathbf
  A_n^{-1} X_i\right) = -\infty, ~\d_\sC\left(\mathbf A_n^{-1} \bar
  X_n\right) = 0\right)\\ &\qquad+ \bbP \left( \min_{1\leq i\leq n}\d_B\left(\mathbf A_n^{-1} X_i \right) = -\infty,~\max_{1 \leq i\leq n} \d_\sC\left( \mathbf
  A_n^{-1} X_i\right) = -\infty,~\d_\sC\left(\mathbf A_n^{-1} \bar
  X_n\right) = 0\right)\\
&\leq \bbP\left(\min_{1\leq i\leq n}\d_B\left(\mathbf A_n^{-1} X_i \right) = -\infty,~\max_{1 \leq i\leq n} \d_\sC\left( \mathbf
  A_n^{-1} X_i\right) = -\infty \right),
  \end{align*}
  where the first term of the first display has measure zero since it
  is impossible that all the sample points are in the set $B$ (i.e.,
  $\min_{1\leq i\leq n} \d_B(\mathbf A_n^{-1} X_i) = 0$) and yet the
  componentwise maximum is in $\sC$ (i.e., $\d_\sC\left(\mathbf A_n^{-1} \bar
    X_n\right) = 0$). Now, Lemma~\ref{lem:convex}~and~\eqref{eq:43} together imply that
\begin{align*}
  \limsup_{n\to\infty} \bbP\left(\min_{1\leq i\leq n}\d_B\left(\mathbf A_n^{-1} X_i \right) = -\infty,~\max_{1 \leq i\leq n} \d_\sC\left( \mathbf
  A_n^{-1} X_i\right) = -\infty \right) = 0.
\end{align*}
Therefore, it follows automatically that
\begin{align*}
  \limsup_{n\to\infty} \frac{1}{\|\mathbf A_n\|_2^2} \log \bbP\left( \max_{1 \leq i\leq n} \d_\sC\left( \mathbf
  A_n^{-1} X_i\right) = -\infty, \d_\sC\left(\mathbf A_n^{-1} \bar
  X_n\right) = 0\right) = -\infty,
\end{align*}
provided the convergence in~\eqref{eq:43} and~\eqref{eq:44} is
$\omega(\log n)$, completing the proof. Here, $f(n) \in \omega(g(n))$
is taken to mean that for all $k > 0$ and large enough $n$, $|f(n)| \geq k |g(n)|$.

Now, to see that the convergence in~\eqref{eq:43} (and
hence~\eqref{eq:44}) is `fast enough,' note that
\begin{align*}
&\bbP \left( \min_{1 \leq i\leq n}\d_B\left( \mathbf A_n^{-1} X_i\right) = 0 ~\bigg| \max_{1 \leq i \leq n} \d_\sC\left(\mathbf A_n^{-1} X_i
    \right) = -\infty\right) = \frac{\bbP\left(\min_{1\leq
                                 i\leq n} \d_B\left( \mathbf A_n^{-1}
                                 X_i\right) = 0\right)}{\bbP\left( \max_{1 \leq i \leq n} \d_\sC\left(\mathbf A_n^{-1} X_i
    \right) = -\infty\right)}.
\end{align*}

Taking the logarithm, we obtain
\begin{align*}
  \log \bbP&\left(\min_{1 \leq i\leq n} \d_B\left( \mathbf
  A_n^{-1} X_i\right) = -\infty ~\bigg | \max_{1 \leq i \leq n}\d_\sC\left(\mathbf A_n^{-1} \bar
  X_i\right) = -\infty\right)\\ \qquad \qquad &= \log\left(1 - \frac{\bbP\left(\min_{1\leq
                                 i\leq n} \d_B\left( \mathbf A_n^{-1}
                                 X_i\right) = 0\right)}{\bbP\left( \max_{1 \leq i \leq n} \d_\sC\left(\mathbf A_n^{-1} X_i
                                \right) = -\infty\right)}\right)
\end{align*}
Now, the i.i.d. assumption implies that $\bbP\left( \min_{1\leq
                                 i\leq n} \d_B\left( \mathbf A_n^{-1}
                                 X_i\right) = 0 \right) = \bbP(X_1
                             \leq \mathbf A_n x^*)^n$, and 
                             Lemma~\ref{prop:convex} yields
                             $
                               \bbP\left( X_1 \leq \mathbf A_n x^*
                               \right) \asymp \left( 1 - \exp(- \a_n
                               \log n)\right),
                             $
                             ~where $\a_n =\frac{1}{2} \langle x^*,
                             \mathbf A_n^{-1} \Sigma^{-1} \mathbf
                             A_n x^*\rangle$. Similarly,
                             Lemma~\ref{lem:convex} implies that
                             \begin{align*}
                               \bbP\left( \max_{1 \leq i \leq n} \d_\sC\left(\mathbf A_n^{-1} X_i
    \right) = -\infty\right) &= 1 - \bbP\left( \max_{1 \leq i \leq n}
                               \d_\sC(\mathbf A_n^{-1} X_i) \geq 0
                               \right)  \\ &\asymp 1 - \exp((\a_n-1)\log n).
                             \end{align*}
                             Note that we have implicitly used
                             Assumption~\ref{assume:3} in the
                             preceding approximations. Furthermore, by Assumption~\ref{assume:3} it follows
                             that $\a_n \to \a := \frac{1}{2}\langle
                             x^*, \mathbf A^{-1} \Sigma^{-1} \mathbf A
                             x^*\rangle > 1$ as $n \to \infty$. Therefore, we have
                             \begin{align*}
                                 \log \bbP\left(\min_{1 \leq i\leq n} \d_B\left( \mathbf
  A_n^{-1} X_i\right) = -\infty ~\bigg | \max_{1 \leq i \leq n}\d_\sC\left(\mathbf A_n^{-1} \bar
  X_i\right) = -\infty\right) \asymp \log \left( 1 - \frac{\left(1 -
                                            n^{-\a_n} \right)^n}{1 -
                                            n^{\a_n - 1}}\right).
                             \end{align*}
                             It can be seen that
\begin{align*}
   \log \left( 1 - \frac{\left(1 - n^{-\a_n} \right)^n}{1 -n^{\a_n -
  1}}\right) & =\log \left(n^{\a_n n} (n^{\a_n-1}-1) - (n^{\a_n}-1)^n
               n^{\a_n-1}\right) - \log \left( n^{\a_n n}(n^{\a_n-1} -
               1) \right)\\
  & \in \omega(\log n),
\end{align*}
thereby completing the proof.
\end{proof}

\section{An rLDP for Gaussian Extrema on Polyhedral Sets}

In this section, we derive the asymptotic likelihood that the componentwise maximum $M_n$ lies in an atypical polyhedral set, as depicted in
Figure~\ref{fig:2}. Polyhedral atypical sets arise quite often in
statistical studies of rare events --- for instance~\cite{DeSi1999}
study the estimation of rare events, and present an application where
the ``failure" region is a halfspace, a special case of the results in this section.

The main result in this section will utilize an rLDP for Gaussian
extremes on \emph{block sets} of the type $\sC := \{x \in \bbR^d:
e_j^Tx^* \geq e_j^Tx^*, j=1,2,\ldots,d\}$. Block sets have been considered recently~\cite{Has2005,HasHu2003} in establishing tail probability bounds for a single Gaussian random vector, i.e., in establishing bounds for $\bbP(X > x)$.

In order to establish the rLDP for a block set, we first observe that
\begin{align*}
  \{\d_\sC(M_n) \geq 0\} &=\{M_n \in \sC := \{y \in \bbR^d :
                     y > x\}\}\\
  &= \{M_n > x\}.
\end{align*}
The result then follows as a simple consequence of Theorem~\ref{thm:5}.

\begin{proposition} \label{prop:boxes}
  Suppose Assumption~\ref{assume:3} and Assumption~\ref{assume:3c} are
  satisfied. Then,
  \begin{align}
       \lim_{n\to\infty} \frac{1}{\|\mathbf A_n\|^2_2} \log \bbP
    \left(M_n > \mathbf A_n x\right)
    = \frac{1}{2} -\frac{1}{2} {\left\langle x , \mathbf A^{-1} \Sigma^{-1} \mathbf A x \right \rangle}.
  \end{align}
\end{proposition}
\begin{proof}
Observe that $x = \arg\,\inf_{y \in \sC} \frac{1}{2} \langle y,
\mathbf A^{-1} \Sigma^{-1} \mathbf A y\rangle$. The final
expression follows automatically from Theorem~\ref{thm:5}.
\end{proof}

Towards constructng an rLDP for polyhedral sets, consider the $m\times d$ matrix
\begin{align*}
  \mathbf B = 
  \begin{bmatrix}
    & -b_1^T-& \\
    & \vdots& \\
    & -b_m^T- & 
  \end{bmatrix}.
\end{align*} Analogous to Assumption~\ref{assume:3}, we
suppose that there exists a sequence of matrices 
\begin{align*}
\mathbf B_n = \begin{bmatrix}
    & -b_{n,1}^T-& \\
    & \vdots& \\
    & -b_{n,m}^T- & 
  \end{bmatrix},
\end{align*}
and vectors $\mathbf c_n=(c_{n,1},\ldots,c_{n,m})$ such that
\begin{assumption}~\label{assume:4b}
  \begin{align}
    \label{eq:19}
    \left(\frac{\mathbf B_n}{\|\mathbf B_n\|}, \frac{\mathbf c_n}{\|\mathbf
    B_n\|} \right) &\to \left(\mathbf B,\mathbf c \right),\\
    \frac{\|\mathbf B_n\|^2_2}{2\log n} &\to 1, ~\text{as}~n\to\infty.
  \end{align}
\end{assumption} We are then interested in the
event 
\begin{align}
\left\{ M_n = \max_{1 \leq i \leq n} X_i \in \mathcal C^n_{\mathbf B_n}
  \right\} &= \bigcap_{j=1}^d \left\{b_{n,j}^T M_n >
             c_{n,j}  \right\}\\
  &= \left\{ \mathbf B_n M_n > \mathbf c_n \right\}, \label{eq:15}
\end{align}
where $\mathcal
C_{\mathbf B_n}$ is a convex open set with boundary determined by the
piecewise linear curve
$\partial \mathcal C^n_{\mathbf B_n}(x) = \max_{1 \leq j \leq m} \{b_{n,j}^T x
- c_{n,j}\}$ for all $x \in \mathbb R^d$. \input{hyper} For instance,
Figure~\ref{fig:2} portrays a two-dimensional example ($d=2$) with
the convex set determined by $m=2$ piecewise linear segments (in
Figure~\ref{fig:2a}) or by $m=3$ piecewise linear segments (in Figure~\ref{fig:2b}).

First consider the case where the the matrix $\mathbf B_n$ has full
column rank, that is, where $B_n$ has $d$ columns. Consequently, the
left pseudo-inverse matrix $\mathbf B_{n,+}^{-1} := (\mathbf B_n^T \mathbf
B_n)^{-1} \mathbf B_n^T$ exists, and the linear system $\mathbf
B_n x = \mathbf c_n$ has a unique solution, which can be exploited to
easily establish the rLDP as a consequence of Theorem~\ref{thm:5}.

On the other hand, when $\mathbf B_n$ has less than full rank  the ``projected'' vectors
$\{\mathbf B_n X_1, $ $\ldots, \mathbf B_n X_n\}$ lie in a subspace of at most
$d$ dimensions almost surely, with {\it singular} covariance matrix
$\Sigma_{\mathbf B_n} = \mathbf B_n \Sigma \mathbf B_n^T$. Furthermore, the
rank of the matrix $\mathbf B_n$ is at most $d$, implying that the
subspace has zero Lebesgue measure.
In
Figure~\ref{fig:2b}, for example, the projected random vectors lie in
a two-dimensional subspace.

The boundary $\partial \mathcal C^n_{\mathbf B_n}$ can also be defined through
the intersection of hyperplanes. Given $m$ hyperplanes, defined
through the rows of $\mathbf B_n$, there exist $m-1$
matrices $\mathbf B_n^1, \ldots \mathbf B_n^{m-1}$ that have full column
rank where
\begin{align} \label{eq:intersect-matrix}
\mathbf B_n^i :=
  \begin{bmatrix}
    & -b_{n,i}^T-&\\
    & -b_{n,i+1}^T&
  \end{bmatrix},
\end{align}
and $c_n^i := (c_{n,i}, c_{n,i+1})$ for all $i \in \{1,\ldots,m-1\}$ . Observe that for any 
random vector, 
\begin{align}
  \nonumber
  \left\{ X \in \mathcal C^n_{\mathbf B_n} \right\} &= \left\{ \mathbf B_n^1
  X > c^1_{n}, \mathbf B_n^2 X > c_{n}^2, \cdots, \mathbf B_n^{m-1} X > c_{n}^{m-1}
                                                  \right\}\\
\label{eq:non-full-rank}
&= \left\{X > (\mathbf B_n^1)^{-1} c^1_n \wedge \cdots\wedge(\mathbf B_n^{d-1})^{-1} c_n^{m-1}\right\}.
\end{align}
where $\wedge$ represents the minimizaton operator, and
the inverse matrices exist since $\{\mathbf B_n^i\}$ are full rank. This
fact will be exploited to establish the main theorem of this section.

\begin{proposition}\label{thm:4}
  Suppose $\mathbf B_n$ and $\mathbf c_n$ satisfy
  Assumption~\ref{assume:4b}. 
  \begin{enumerate}
    \item[(i)] If $\mathbf B_n$ and $\mathbf B$ have full column rank, then
      \begin{align} 
        \lim_{n\to\infty} \frac{1}{\left\| \mathbf B_n\right\|^2_2} \log
        \mathbb P  \left( \max_{1\leq i \leq n} X_i \in \mathcal C^n_{\mathbf B_n}\right) = \frac{1}{2} - \frac{1}{2} z^T \Sigma^{-1} z,
    \end{align}
    where $z = \mathbf B_{+}^{-1} c$.

    \item[(ii)] If $\mathbf B_n$ and $\mathbf B$ do not have full column rank, then
      \begin{align}
        \label{eq:16}
               \lim_{n\to\infty} \frac{1}{\left\| \mathbf B_n\right\|^2_2} \log
        \mathbb P  \left( \max_{1\leq i \leq n} X_i \in \mathcal C^n_{\mathbf B_n}\right) = \frac{1}{2} - \frac{1}{2} z^T \Sigma^{-1} z,
      \end{align}
      where $z := (\mathbf B^1)^{-1} c^1 \wedge\cdots\wedge(\mathbf
      B^{d-1})^{-1} c^{m-1}$, and the matrices $\{\mathbf B_1,\cdots,
      \mathbf B_d\}$ and vectors $\{c^1,\ldots,c^{m-1}\}$ are defined analogous to~\eqref{eq:intersect-matrix}.
  \end{enumerate}
\end{proposition}

\begin{proof}
  (i) 
  Recall that $M_n = \max_{1\leq i\leq n} X_i$. Since $\mathbf B_n$ is full rank
  \begin{align*}
    \bbP\left( M_n\in \mathcal C^n_{\mathbf B_n}
    \right) &= \mathbb P \left( \mathbf B_n M_n > \mathbf
  \mathbf c_n \right)\\ &=  \mathbb P \left( M_n > \mathbf B_{n,+}^{-1}
                            \mathbf c_n\right).
  \end{align*}
  Now, since
  linear transforms between finite dimensional spaces are continuous, it follows that
  \begin{align}\label{eq:lim-inverse}
    \lim_{n \to \infty} \mathbf B_{n,+}^{-1} \mathbf c_n \to \mathbf
    B^{-1}_+ c = z.
  \end{align}
  Therefore, following Proposition~\ref{prop:boxes} we obtain
  \begin{align}
    \lim_{n\to\infty} \frac{1}{\|\mathbf B_n\|^2} \log \mathbb P
     \left(\max_{1\leq i\leq n} X_i \in \mathcal C^n_{\mathbf B_n} \right)
     =\frac{1}{2} - \frac{1}{2} z^T \Sigma^{-1} z.
  \end{align}

  (ii) The proof of this part of the theorem follows that of part (i)
  and using the definition of $\{X \in \mathcal C^n_{\mathbf B_n}\}$
  in~\eqref{eq:non-full-rank}. We skip the details.
\end{proof}

A straightforward corollary of Proposition~\ref{thm:4} is to assume that
$\mathbf B_n = \mathbf B$ and $\|B\| = 1$. In this case, the sequence of convex
regions $\mathcal C^n_{\mathbf B_n}$ are determined entirely by the
sequence of vectors $\mathbf c_n$, and each $\sC^n_{\mathbf B_n}$ is a
scaled version of the limit region $\sC_{\mathbf B}$. A more
interesting corollary emerges when $\sC_{\mathbf B}$ is an open rhomboidal
region in the positive orthant. Specifically, we are interested
in the $\ell_1$-ball centered at some point $x_0$ in the orthant. Observe that the hyperplanes closest to
the origin defining the rhombus are represented by an orthonormal
matrix $\mathbf B$. We obtain the following corollary to
Theorem~\ref{thm:4}(i).

\begin{corollary}
  Let $\mathbf B$ be an orthonormal matrix. Assume that there exists a
  sequence of matrices $\mathbf B_n$ and vectors $\mathbf c_n$ that satisfy
  Assumption~\ref{assume:4b}, along with $\mathbf B$ and vector
  $\mathbf c$. Then,
  \begin{align*}
    \lim_{n\to\infty} \frac{1}{\|\mathbf B_n\|^2_2} \log \bbP \left(
    \max_{1 \leq i \leq n} X_i \in \sC^n_{\mathbf B_n} \right) = \frac{1}{2} -
    \frac{1}{2}z^T \Sigma^{-1} z
  \end{align*}
  where $z = \mathbf B^{-1} \mathbf c$. 
\end{corollary}

\section{An rLDP for Gaussian-Mixture Extrema on Convex Sets}

Recall that a Gaussian-mixture model $\tilde{X}$ is defined as a multivariate
random vector whose density function satisfies
\begin{align}
  \label{eq:gmm}
  \bbP(\tilde{X} \in dx) = \sum_{i=1}^K \pi_j \phi_j(dx),
\end{align}
where $K$ is the number of mixture components, $\{\pi_1,\ldots,\pi_K\}$ is the set of mixture probabilities
such that $\sum_{j=1}^K \pi_j = 1, \pi_j \geq 0$ and
$\{\phi_1, \ldots,\phi_K\}$ is a set of multivariate Gaussian density
functions. The random vaector $\tilde{X}$ is not necessarily Gaussian; nonetheless, as will be demonstrated, the previous rLDP for
Gaussian extrema can be used to derive a corresponding result for Gaussian-mixture extrema. We focus on the general convex case, and leave
specific examples to the reader.

The parameter set of the Gaussian mixture model consists of the
tuples, $\{(\pi_j, \m_j, \Sigma_j),~j=1,\ldots,K\}$, where $\mu_j$ and
$\Sigma_j$ are the mean vector and covariance matrix of the $j$th
component Gaussian. Without loss of generality we assume that $\mu_j >
0$, and let $\sC$ be an atypical convex set in the positive
orthant. We will make the following further
\begin{assumption}
  \label{assume:4}
  The component mean vectors $\{\m_1,\ldots,\m_K\}$ do not belong to
  the closed convex set $\sC$, that is, $\mu_j \notin \sC$ for $j=1,2,\ldots,K$.
\end{assumption}
\begin{assumption} \label{assume:3d}
  The component mean vectors $\{\m_1,\ldots,\m_K\}$ and covariance
  matrices $\{\Sigma_1,\ldots,\Sigma_K\}$ satisfy, for any $\e \in \sC$,
  \begin{align}
    \min_{1\leq j\leq K} \frac{1}{2}\left\langle (\e - \m_j), \mathbf A^{-1} \Sigma_j^{-1}
    \mathbf A (\e - \m_j) \right\rangle > 1.
  \end{align}
\end{assumption}
We are now ready to establish the rLDP for the Gaussian-mixture extrema on closed convex sets.

\begin{proposition}\label{prop:gmm}
    
  Let $\tilde{X}_i \in \bbR^d, i=1,2,\ldots$ be iid Gaussian-mixture random vectors each having mixture probabilties $\{\pi_1,\ldots,\pi_K\}$ and mixture component densitiies 
$\{\phi_1, \ldots,\phi_K\}$. Suppose Assumption~\ref{assume:3}, Assumption~\ref{assume:4} and Assumption~\ref{assume:3d} are
  satisfied. Then the componentwise maximum $$\tilde{M}_n := (\max_{1\leq i\leq n} \tilde{X}_i^1, \ldots, \max_{1\leq i\leq n} \tilde{X}_i^d)$$ satisfies an rLDP with rate $J(\sC) : = \frac{1}{2} - \frac{1}{2} \bigwedge_{j=1}^K \inf_{x \in \sC}\langle (x-\m_j), \mathbf
    A^{-1} \Sigma_j^{-1} \mathbf A (x-\m_j)\rangle$ and speed $v(a) = a^2$, that is,
  \begin{align}
    \lim_{n \to \infty} \frac{1}{\|\mathbf A_n\|_2^2} \log \bbP
    \left(\d_\sC\left(\mathbf A_n^{-1} \tilde{M}_n\right) \geq 0\right) =
    \frac{1}{2} - \frac{1}{2} \bigwedge_{j=1}^K \inf_{x \in \sC}\langle (x-\m_j), \mathbf
    A^{-1} \Sigma_j^{-1} \mathbf A (x-\m_j)\rangle.
  \end{align}
\end{proposition}
\begin{proof}
  First consider a single Gaussian mixture random vector, $X$, and observe that
  \begin{align*}
    \bbP\left( \d_\sC\left(\frac{X}{a_n}\right) \geq 0\right) =
    \sum_{j=1}^K \pi_j \int_{a_n \sC} \phi_j(dx),
  \end{align*}
  using~\eqref{eq:gmm}. By the ``principle of the largest term"
  \cite[Lemma 1.2.15]{DeZe2010} it
  follows that
  \begin{align*}
    \limsup_{n\to\infty} \frac{1}{a_n^2} \log \bbP\left(
    \d_\sC\left(\frac{X}{a_n}\right) \geq 0\right) = \bigvee_{j=1}^K
    \limsup_{n\to\infty} \frac{1}{a_n^2} \log \left( \pi_j \int_{a_n
    \sC} \phi_j(dx)\right). 
  \end{align*}
  Following Lemma~\ref{prop:convex}, $\limsup_{n\to\infty} \frac{1}{a_n^2} \log \left( \pi_j \int_{a_n\sC} \phi_j(dx)\right) = -\frac{1}{2} \inf_{x
  \in\sC}\langle(x-\m_j), \Sigma_j^{-1} (x-\m_j) \rangle$. (Note that
Lemma~\ref{prop:convex} is stated for a centered Gaussian
multivariate, but the proof does not change with a non-centered
distribution.) 

It follows that
\begin{align}
  \label{eq:gmm-ub-single}
   \limsup_{n\to\infty} \frac{1}{a_n^2} \log \bbP\left(
    \d_\sC\left(\frac{X}{a_n}\right) \geq 0\right) = -\frac{1}{2}
  \bigwedge_{j=1}^K \inf_{x\in\sC}\left\langle(x-\m_j), \Sigma_j^{-1} (x-\m_j) \right\rangle.
\end{align} Next, we use the obvious lower bound and obtain,
\begin{align*}
 \log \bbP\left( \d_\sC\left(\frac{X}{a_n}\right) \geq 0\right) & \geq
                                                                  \log
                                                                  \bigvee_{j=1}^K \left( \pi_j \int_{a_n
                                                                  \sC} \phi_j(dx)\right)\\
  &=\bigvee_{j=1}^K \log \left( \pi_j \int_{a_n
                                                                  \sC} \phi_j(dx)\right),
\end{align*}
where the equality follows from the monotonicity of the logarithm
function. Again, applying Lemma~\ref{prop:convex} to the lower bound
we obtain
\begin{align}
  \label{eq:gmm-lb-single}
   \liminf_{n\to\infty} \frac{1}{a_n^2} \log \bbP\left(
    \d_\sC\left(\frac{X}{a_n}\right) \geq 0\right) \geq -\frac{1}{2}
  \bigwedge_{j=1}^K \inf_{x
  \in\sC}\left\langle(x-\m_j), \Sigma_j^{-1} (x-\m_j) \right\rangle.
\end{align}

For the sample maximum of the Gaussian-mixture observe that
Lemma~\ref{lem:convex} implies that
\begin{align*}
  \lim_{n\to\infty} \frac{1}{a_n^2} \log \bbP\left(
  \max_{1 \leq i \leq n}\d_\sC\left(\frac{X_i}{a_n}\right) \geq 0
  \right) = \frac{1}{2} - \frac{1}{2} \bigwedge_{j=1}^K \inf_{x
  \in\sC}\left\langle(x-\m_j), \Sigma_j^{-1} (x-\m_j) \right\rangle.
\end{align*}
An application of Lemma~\ref{lem:exp-eq} yields the final result.
\end{proof}


%% file: lbconvex.tex
\tikzset{mystyle/.style={shape=circle,fill=red,scale=0.4}}
\begin{figure}[h]\centering
  \begin{tikzpicture}[
    thick, >=stealth', dot/.style =
      { draw, fill = none, circle, inner sep = 0pt, minimum size = 5pt
      },scale=0.5]
      \coordinate(0) at (0,0);
      \draw[->] (-6,0) -- (6,0) node[below] {$Z_1$};
      \draw[->] (0,-.3) -- (0,6) node[left] {$Z_2$};
      \draw[dashed] (-6,1.5) -- (6,1.5);
      \draw[dashed] (-6,5.1) -- (6,5.1);
      \draw[dashed,<->] (-5,1.5) -- (-5,5) node[midway,left]{$\|z^*\| \e$};
            
       \node[below] at (.6,1.75){ $\pmb z^*$};
       \node [anchor=west] (c) at (-4,6) {\Large $L\sC$};
	   \node [anchor=west] (cs) at (.5,4) {\Large $\tilde{\sC}$};
	   \node[mystyle] at (0,1.5){};
			\draw[scale=0.5,domain=-10:10,smooth,variable=\x,red] plot ({\x},{0.1 * (\x* \x)+3});     
			\draw[scale=0.5,domain=-3.8:3.8,smooth,variable=\x,green,pattern=north west lines, pattern color=blue] plot ({\x},{0.5*(\x*\x)+3});          
           \draw[rotate=20] (0,0) circle (0.2cm);
            \draw[rotate=20] (0,0) circle (0.5cm);
  \end{tikzpicture}
  \caption{Bivariate representation of set $\tilde{\sC}$, the intersection of epigraph of $L\sC$ and halfspace. }
  
  \label{fig:pareto}
\end{figure}
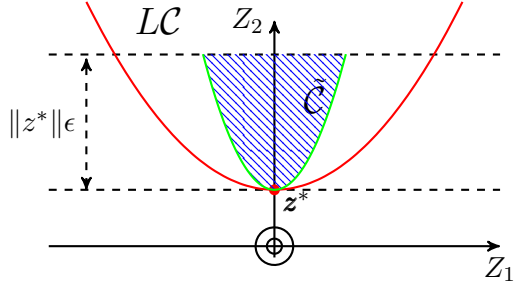

%% file: diffcone.tex
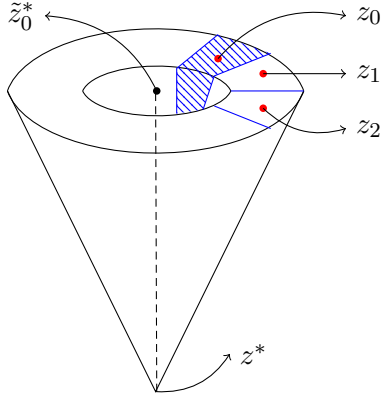
\begin{figure}[t]
  \centering
   \begin{tikzpicture}
    \draw (0,0) arc (170:10:2cm and 1cm)coordinate[pos=0] (a);
    \draw (0,0) arc (-170:-10:2cm and 1cm)coordinate (b);
    \draw (1,0) arc (170:10:1cm and 0.4cm)coordinate[pos=0] (c);
    \draw (1,0) arc (-170:-10:1cm and 0.4cm)coordinate (d);
    \draw[densely dashed] ([yshift=-4cm]$(c)!0.5!(d)$)  -- node[right,font=\footnotesize] {} coordinate[pos=0.95] (aa)($(a)!0.5!(b)$);
    \fill[black] ($(c)!0.5!(d)$) circle [radius=.05];
    \draw[->] ($(c)!0.5!(d)$) to[bend right] ([yshift=1cm]$(a)!0.5!(c)$) node[left] {$\tilde z_0^*$};
    \draw (a) -- ([yshift=-4cm]$(a)!0.5!(b)$) node[right]{} coordinate[pos=0.95]-- (b);
    \draw[->] ([yshift=-4cm]$(a)!0.5!(b)$) to[bend right] ([yshift=-3.5cm]$(a)!0.75!(b)$) node[right] {$z^*$};
    \draw[blue] (2.75,-0.2) -- (3.5,-.5);
    \draw[blue] (d) -- (b);
    \draw[blue] (2.6,-0.25)--(2.75,0.2) -- (3.5,.5);
    \draw[blue] (2.25,-0.3)--(2.25,0.3) -- (2.8,0.75);
    \fill[pattern=north west lines, pattern color=blue] (2.6,-0.25)--(2.75,0.2)--(3.5,.5)--(2.8,0.75)--(2.25,0.3)-- (2.25,-0.3);
    \fill[red] (2.8,0.43) circle [radius=.05];
    \draw[->] (2.8,0.43) to[bend left] (4.5,1.0) node[right] {$z_0$};
    \fill[red] (3.4,0.23) circle [radius=.05];
    \draw[->] (3.4,0.23) to (4.5,0.23) node[right] {$z_1$};
    \fill[red] (3.4,-0.23) circle [radius=.05];
     \draw[->] (3.4,-0.23) to[bend right] (4.5,-0.5) node[right] {$z_2$};
  \end{tikzpicture}
  \caption{The truncated hoop cone $\mathcal{TC}(z^*,z_0,\e/2,\e/4)$ is the blue shaded region. We construct $(n+1)$ truncated hoop cones. These are ``identical" symmetric objects about the $d$th axis.}
\end{figure}\label{hoop-cones}

%% file: hyper.tex
\begin{figure}[h]\centering
  \begin{subfigure}[b]{0.45\textwidth}
  \begin{tikzpicture}[
    thick, >=stealth', dot/.style =
      { draw, fill = none, circle, inner sep = 0pt, minimum size = 10pt
      }]
      \coordinate(0) at (0,0);
      \draw[->] (-0.3,0) -- (6,0) node[below] {$X_1$};
      \draw[->] (0,-0.3) -- (0,4) node[left] {$X_2$};
      \draw[-,dashed] (0,3) -- (2,0); \draw[color=red] (0,3) -- (5/4,9/8) coordinate[midway] (M1);
      \draw[-,dashed] (0,1.5) -- (5,0); \draw[color=red] (5/4,9/8) -- (5,0) coordinate[midway] (M2);
      \fill[pattern = north west lines,pattern color=blue,dashed] (0,4) -- (0,3) -- (5/4,9/8) -- (5,0) -- (6,0) -- (6,4) -- (0,4);\
      \draw[->] ($(M1)!0.5cm!90:(0,3)$) -- ($(M1)!0.5cm!270:(0,3)$) node[above] {$b_1$};
      \draw[->] ($(M2)!0.5cm!270:(5,0)$) -- ($(M2)!0.5cm!90:(5,0)$) node[above] {$b_2$};
  \end{tikzpicture}
  \caption{Full rank case. }
  \label{fig:2a}
  \end{subfigure}
  ~
  \begin{subfigure}[b]{0.45\textwidth}
    \begin{tikzpicture}[
    thick, >=stealth', dot/.style =
      { draw, fill = none, circle, inner sep = 0pt, minimum size = 10pt
      }]
      \coordinate(0) at (0,0);
      \draw[->] (-0.3,0) -- (6,0) node[below] {$X_1$};
      \draw[->] (0,-0.3) -- (0,4) node[left] {$X_2$};
      \draw[-,dashed] (0,3.5) -- (1.5,0); \draw[color=red] (0,3.5) -- (.75,1.75) coordinate[midway] (M1);
      \draw[-,dashed] (0,2.5) -- (2.5,0); \draw[color=red] (.75,1.75) -- (15/8,5/8) coordinate[midway] (M2);
      \draw[-,dashed] (0,1.0) -- (5,0); \draw[color=red] (15/8,5/8) -- (5,0) coordinate[midway] (M3);
      \fill[pattern = north west lines,pattern color=blue,dashed] (0,3.5) -- (0.75,1.75) -- (15/8,5/8) -- (5,0) -- (6,0) -- (6,4) -- (0,4);
      \draw[->] ($(M1)!0.5cm!90:(0,3.5)$) -- ($(M1)!0.5cm!270:(0,3.5)$) node[above] {$b_1$};
      \draw[->] ($(M2)!0.5cm!270:(15/8,5/8)$) -- ($(M2)!0.5cm!90:(15/8,5/8)$) node[above] {$b_2$};
      \draw[->] ($(M3)!0.5cm!270:(5,0)$) -- ($(M3)!0.5cm!90:(5,0)$) node[above] {$b_3$};
  \end{tikzpicture}
    \caption{Over-determined case. }
    \label{fig:2b}
  \end{subfigure}
  \caption{Examples of the open convex set $\mathcal C_{\mathbf B}$ in
    two dimensions.}
  \label{fig:2}
\end{figure}
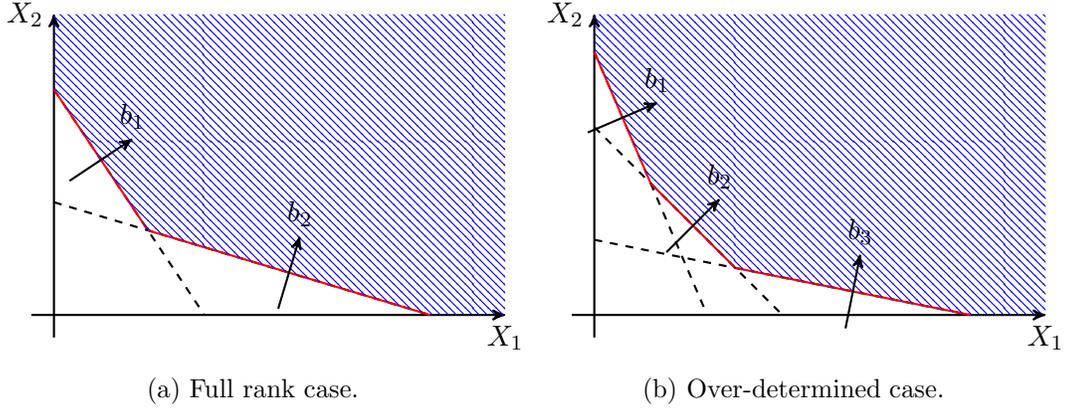

%% file: arxivVersion.bbl
\begin{thebibliography}{10}

\bibitem{2007asmgly}
{\sc Asmussen, S. and Glynn, P.~W.} (2007).
\newblock {\em Stochastic Simulation: Algorithms and Analysis}.
\newblock Springer, New York, NY.

\bibitem{2004boyvan}
{\sc Boyd, S. and Vandenberghe, L.} (2004).
\newblock {\em Convex Optimization}.
\newblock Cambridge University Press, Cambridge, U.K.

\bibitem{ChPa2018}
{\sc Chen, J. and Pasupathy, R.} (2018).
\newblock The efficient estimation of large exceedances of functions of
  elliptical random vectors.
\newblock In Preparation.

\bibitem{Co1982b}
{\sc Cohen, J.~P.} (1982).
\newblock Convergence rates for the ultimate and pentultimate approximations in
  extreme-value theory.
\newblock {\em Adv. Appl. Probab.\/} {\bf 14,} 833--854.

\bibitem{Co1982}
{\sc Cohen, J.~P.} (1982).
\newblock The penultimate form of approximation to normal extremes.
\newblock {\em Adv. Appl. Probab.\/} {\bf 14,} 324--339.

\bibitem{DeFe2007}
{\sc De~Haan, L. and Ferreira, A.} (2007).
\newblock {\em Extreme value theory: an introduction}.
\newblock Springer Science \& Business Media.

\bibitem{DeHo1972}
{\sc de~Haan, L. and Hordijk, A.} (1972).
\newblock The rate of growth of sample maxima.
\newblock {\em Ann. Math. Stat.\/} 1185--1196.

\bibitem{DeSi1999}
{\sc de~Haan, L. and Sinha, A.~K.} (1999).
\newblock Estimating the probability of a rare event.
\newblock {\em Ann. Stat.\/} {\bf 27,} 732--759.

\bibitem{de1983}
{\sc de~Oliveira, J.~T.} (1983).
\newblock {\em Statistical extremes and applications} vol.~131.
\newblock Springer Science \& Business Media.

\bibitem{DeHaJiRo2018}
{\sc D{\c e}bicki, K., Hashorva, E., Ji, L. and Rolski, T.} (2018).
\newblock Extremal behavior of hitting a cone by correlated brownian motion
  with drift.
\newblock {\em Stoch. Proc. Appl.\/}.

\bibitem{DeHaJiTa2015}
{\sc D{\c e}bicki, K., Hashorva, E., Ji, L. and Tabi{\'s}, K.} (2015).
\newblock Extremes of vector-valued gaussian processes: Exact asymptotics.
\newblock {\em Stoch. Proc. Appl.\/} {\bf 125,} 4039--4065.

\bibitem{DeKoMaRo2010}
{\sc D{\c{e}}bicki, K., Kosi{\'n}ski, K.~M., Mandjes, M. and Rolski, T.}
  (2010).
\newblock Extremes of multidimensional gaussian processes.
\newblock {\em Stoch. Proc. Appl.\/} {\bf 120,} 2289--2301.

\bibitem{DeZe2010}
{\sc Dembo, A. and Zeitouni, O.}
\newblock Large deviations techniques and applications, volume 38 of stochastic
  modelling and applied probability 2010.

\bibitem{DuLeSu2003}
{\sc Duffy, K., Lewis, J.~T. and Sullivan, W.~G.} (2003).
\newblock Logarithmic asymptotics for the supremum of a stochastic process.
\newblock {\em Ann. Appl. Probab.\/} {\bf 13,} 430--445.

\bibitem{FiTi1928}
{\sc Fisher, R.~A. and Tippett, L. H.~C.} (1928).
\newblock Limiting forms of the frequency distribution of the largest or
  smallest member of a sample.
\newblock In {\em Mathematical Proceedings of the Cambridge Philosophical
  Society}.
\newblock vol.~24 Cambridge University Press.
\newblock pp.~180--190.

\bibitem{GiMa2014}
{\sc Giuliano, R. and Macci, C.} (2014).
\newblock Large deviation principles for sequences of maxima and minima.
\newblock {\em Commun. Stat. A-Theor.\/} {\bf 43,} 1077--1098.

\bibitem{Gn1943}
{\sc Gnedenko, B.} (1943).
\newblock Sur la distribution limite du terme maximum d'une serie aleatoire.
\newblock {\em Ann. Math.\/} 423--453.

\bibitem{Ha1979b}
{\sc Hall, P.} (1979).
\newblock On the rate of convergence of normal extremes.
\newblock {\em J. Appl. Probab.\/} {\bf 16,} 433--439.

\bibitem{HaWe1979}
{\sc Hall, W. and Wellner, J.~A.} (1979).
\newblock The rate of convergence in law of the maximum of an exponential
  sample.
\newblock {\em Stat. Neerl.\/} {\bf 33,} 151--154.

\bibitem{Has2005}
{\sc Hashorva, E.} (2005).
\newblock Asymptotics and bounds for multivariate gaussian tails.
\newblock {\em J. Theor. Probab.\/} {\bf 18,} 79--97.

\bibitem{HasHu2002}
{\sc Hashorva, E. and H{\"u}sler, J.} (2002).
\newblock On asymptotics of multivariate integrals with applications to
  records.
\newblock {\em Stoch. Models\/} {\bf 18,} 41--69.

\bibitem{HasHu2003}
{\sc Hashorva, E. and H{\"u}sler, J.} (2003).
\newblock On multivariate gaussian tails.
\newblock {\em Ann. I. Stat. Math.\/} {\bf 55,} 507--522.

\bibitem{JuSh2006}
{\sc Juneja, S. and Shahabuddin, P.} (2006).
\newblock Rare-event simulation techniques: an introduction and recent
  advances.
\newblock {\em Handbooks in operations research and management science\/} {\bf
  13,} 291--350.

\bibitem{kim2015guide}
{\sc Kim, S., Pasupathy, R. and Henderson, S.~G.} (2015).
\newblock A guide to sample average approximation.
\newblock In {\em Handbook of simulation optimization}.
\newblock Springer pp.~207--243.

\bibitem{KoMa2015}
{\sc Kosi{\'n}ski, K. and Mandjes, M.} (2015).
\newblock Logarithmic asymptotics for multidimensional extremes under nonlinear
  scalings.
\newblock {\em J. Appl. Probab.\/} {\bf 52,} 68--81.

\bibitem{Ne1983}
{\sc Ney, P.} (1983).
\newblock Dominating points and the asymptotics of large deviations for random
  walk on $r^d$.
\newblock {\em Ann. Probab.\/} {\bf 11,} 158--167.

\bibitem{Ne1984}
{\sc Ney, P.} (1984).
\newblock Convexity and large deviations.
\newblock {\em Ann. Probab.\/} 903--906.

\bibitem{NiShNa2001}
{\sc Nicola, V.~F., Shahabuddin, P. and Nakayama, M.~K.} (2001).
\newblock Techniques for fast simulation of models of highly dependable
  systems.
\newblock {\em IEEE Transactions on Reliability\/} {\bf 50,} 246--264.

\bibitem{Ro2015}
{\sc Rockafellar, R.~T.} (2015).
\newblock {\em Convex analysis}.
\newblock Princeton university press.

\bibitem{Sa1962}
{\sc Savage, I.~R.} (1962).
\newblock Mills' ratio for multivariate normal distributions.
\newblock {\em J. Res. Nat. Bur. Standards Sect. B\/} {\bf 66,} 93--96.

\bibitem{shapiro2013sample}
{\sc Shapiro, A.} (2013).
\newblock Sample average approximation.
\newblock In {\em Encyclopedia of Operations Research and Management Science}.
\newblock Springer pp.~1350--1355.

\bibitem{Sm1982}
{\sc Smith, R.~L.} (1982).
\newblock Uniform rates of convergence in extreme-value theory.
\newblock {\em Adv. Appl. Probab.\/} {\bf 14,} 600--622.

\bibitem{vapnik1998statistical}
{\sc Vapnik, V.} (1998).
\newblock {\em Statistical learning theory. 1998} vol.~3.
\newblock Wiley, New York.

\bibitem{Vi1994}
{\sc Vinogradov, V.} (1994).
\newblock {\em Refined large deviation limit theorems} vol.~315.
\newblock CRC Press.

\bibitem{Vi2015}
{\sc Vivo, P.} (2015).
\newblock Large deviations of the maximum of independent and identically
  distributed random variables.
\newblock {\em Eur. J. Phys.\/} {\bf 36,} 055037.

\end{thebibliography}
